\documentclass[12pt]{article}
\usepackage{amsfonts}
\usepackage{amsmath, amsthm, tikz, fullpage, longtable,ulem,amssymb,mathrsfs}
\usepackage{float}
\usepackage[noblocks]{authblk}
\usepackage{algpseudocode}

\usetikzlibrary{graphs,positioning,decorations.pathreplacing}
\usepackage{fullpage}
\usepackage{amstext}
\usepackage{mathtools}

\newtheorem{theorem}{Theorem}[section]
\newtheorem{lemma}[theorem]{Lemma}
\newtheorem{definition}[theorem]{Definition}
\newtheorem{prop}[theorem]{Proposition}

\newtheorem{cor}[theorem]{Corollary}

\newcommand{\cH}{\mathcal{H}}

\newcommand{\cV}{\mathcal{V}}

\DeclareMathOperator{\dist}{dist}

\DeclareMathOperator{\spn}{span}

\DeclareMathOperator{\rank}{rank}
\DeclareMathOperator{\im}{im}
\DeclareMathOperator{\proj}{proj}

\allowdisplaybreaks

\title{Zero loci of nullvectors and skew zero forcing in graphs and hypergraphs}
\author{Joshua Cooper and Grant Fickes}
\date{\today}

\begin{document}

\maketitle

\begin{abstract}
    There is interesting internal structure in the nullspaces of graph and hypergraph adjacency matrices, especially for trees, bipartite graphs, and related combinatorial classes.  The zero loci of nullvectors, i.e., their zero coordinates' indices, encode information about matchings, coverings, and edges' influence on rank.  This set system is the lattice of flats of a ``kernel matroid'', a subsystem of which are the ``stalled'' sets closed under skew zero forcing (SZF), a graph percolation/infection model known to have connections with rank and nullity. For a wide variety of graphs, the lattice of SZF-closed sets is also a matroid, a fact which can be used to obtain a polynomial-time algorithm for computing the skew zero forcing number.  This contrasts with the general case, where we show that the corresponding decision problem is NP-hard.  We also define skew zero forcing for hypergraphs, and show that, for linear hypertrees, the poset of SZF-closed sets is dual to the lattice of ideals of the hypergraph's nullvariety; while, for complete hypergraphs, the SZF-closed sets and the zero loci of nullvectors are more loosely related.
\end{abstract}

\section{Introduction}

It is classical that the multiplicity of zero as Laplacian eigenvalue of a graph is the number of connected components (see \cite{Chu97}), but the nullity of adjacency matrices is much subtler.  Significant attention has been paid to understanding the adjacency nullity of graphs (an important survey is \cite{GutBor2011}), and to a much lesser extent, hypergraphs (e.g., \cite{CooFic2021}).  The present work is an attempt to understand not just the multiplicity of zero, but its associated nullspace -- or, in the case of hypergraphs, its ``nullvariety''.

In particular, we show that the set of adjacency nullvectors can be decomposed into combinatorially informative components according to their ``zero loci'': the coordinates where the vectors are zero.  Inspired in part by work of Sciriha et al -- for example, \cite{SciMifBor2021} -- we begin by examining the minimal zero locus of trees' nullvectors, which we call their ``generating set''. (Some literature refers to these sets as ``core-forbidden vertices.'')
In Section \ref{Sec:Equiv}, we show that this set has a number of interesting interpretations from the perspective of maximum matchings, vertex covers, the Dulmage-Mendelsohn/Gallai-Edmonds decomposition, the effect on rank of edge deletion and contraction, and {\it skew zero forcing}.

Graph forcing is a topic that emerged from studying the maximum rank and nullity of real matrices in a special class $\mathcal{P}$ (symmetric, skew-symmetric, positive semidefinite, etc) whose nonzero entries correspond to edges of a given graph.  See \cite{HoLiSh22} for an exploration of this rapidly developing topic.  The term ``forcing'' refers to iteratively applying a color-change rule to a unfilled/filled vertex coloring wherein the filled set spreads until it ``stalls''.  (Some literature, e.g., \cite{IMAISU2010}, calls stalled sets ``derived sets'' and refers to unfilled/filled as white/blue or white/black.) The size of the smallest set which only stalls when the whole graph is filled is the ``$\mathcal{P}-$zero forcing number'', and the size of the largest stalled proper subset is the ``failed $\mathcal{P}$-zero forcing number''.  The most common classes considered are ``zero forcing'' (for $\mathcal{P}$ the symmetric matrices) and ``skew zero forcing'' (for $\mathcal{P}$ the skew-symmetric matrices).  For concision, we typically write ``SZF'' for ``skew zero forcing''.

In Section \ref{Sec:Matroids}, we show that the skew zero forcing rule is a closure operator, and the family of SZF-stalled sets is {\it often} the collection of closed sets of a matroid we term the ``SZF matroid.''  Furthermore, the collection of zero loci of nullvectors is always a matroid, the ``kernel matroid'', a quotient of which is the SZF matroid.  We show that these two matroids are identical -- a property we term ``SZF-completeness'' -- for trees, cycles of length divisible by $4$, complete bipartite graphs, and graphs derived by various operations applied to smaller SZF-complete graphs.  We also characterize nonsingular bipartite SZF-complete graphs, extending results of \cite{AnsJacPenSaa2016}, and answer a question of theirs by showing that, while it is NP-hard in general to decide if the skew zero forcing number -- the smallest set which SZF-closes to the full vertex set -- is at most $k$, this quantity can be computed in polynomial time if the set system is a matroid.  In particular, minimal sets whose closure is the whole vertex set are bases, so the rank of the matroid is the SZF number.  (The maximal sets whose closure is {\it not} the whole vertex set -- ``failed sets'' -- are hyperplanes/coatoms.)

In Section \ref{Sec:Hypergraphs}, we extend the story to hypergraphs, giving a new SZF rule for hypergraphs.  (Hogben has studied another choice of rule with nice properties: \cite{Hog2020}.)  We show that SZF-closed sets are a special kind of vertex cover which are in bijection with irreducible components of the nullvariety for linear hypertrees, and describe the nullvarieties' components and SZF-closed families for complete hypergraphs via symmetric polynomials.

Finally, in Section \ref{Sec:Conclusion}, we mention several open problems arising from the present work.\\

Throughout the sequel, we refer to the set of adjacency nullvectors of a graph or hypergraph $G$ by $\ker(G)$.  For graphs, this means that $\ker(G) = \{\mathbf{v} : A(G)\mathbf{v} = \mathbf{0}\}$, where $A(G)$ is the adjacency matrix of $G$; for hypergraphs, we explain the more complicated definition in Section \ref{Sec:Hypergraphs}.  Given a vector $\mathbf{v} = (v_{u_1},\ldots,v_{u_n}) \in \mathbb{C}^{V(G)}$, the set $Z(\mathbf{v}) := \{u \in V(G) : v_u = 0\}$ is the {\it zero locus} of $\mathbf{v}$.

\section{Minimal Zero Loci of Trees}\label{Sec:Equiv}

In this section, we investigate the minimal zero loci of nullvectors of trees.  First, we describe the relationship between edges which are mandatory/optional/forbidden in maximum matchings and vertices which are mandatory/optional/forbidden in minimum vertex covers, via a ``thermal decomposition'' of trees.  In the next subsection, we formally introduce skew zero forcing (``SZF'') and show that it gives rise to a closure operator on vertex sets.  For trees, it turns out that the SZF-closed sets are exactly the zero loci of nullvectors.  In the third subsection, the previous results are connected with the Dulmage-Mendelsohn decomposition in one of our main theorems: a multifaceted characterization of trees' generating sets.  Then, the last subsection gives another description of the thermal decomposition in terms of the effect on rank of deletion or contraction of edges.

\subsection{Matchings, Coverings, and the Thermal Decomposition}

Recall that a {\it matching} of a graph $G$ is a set $F \subseteq E(G)$ of pairwise disjoint edges, and that a {\it cover} is a set $S \subseteq V(G)$ of vertices so that every edge $e \in E(G)$ has a nonempty intersection with $S$.  A cover is {\it minimum} if it has minimum cardinality among all covers, and a matching is {\it maximum} if it has maximum cardinality among all matchings.  A vertex $v \in V(G)$ is a {\it pendant vertex} if it has degree $1$, and an edge is a \textit{leaf edge} if it contains a pendant vertex.  A matching $M$ {\it saturates} a vertex $v$ if there exists $e \in M$ with $v \in e$, and it is a {\it perfect} matching if it saturates all of $V(G)$.

\begin{definition}
Let $T$ be a tree. We say the edge $e \in E(T)$ is \textbf{matching-frozen} if either $e$ is contained in every maximum matching of $T$ or contained in no maximum matching of $T$. If the edge $e$ is not matching-frozen, then we call $e$ \textbf{matching-thawed}, i.e., $e$ is contained in some but not all maximum matchings of $T$.  
\end{definition}

\begin{definition}\label{Def:ThermalDecomp}
Let $T$ be a tree. We say the \textbf{thermal decomposition} of $T$ is a partition of $E(T)$ into three classes, $(M_T, F_T, O_T)$, so that $M_T$ (``mandatory'') is the collection of edges of $T$ in every maximum matching, $F_T$ (``forbidden'') is the collection of edges of $T$ in no maximum matching, and $O_T$ (``optional'') is the collection of edges of $T$ in some but not all maximum matching of $T$. Furthermore, let $F_T' \subseteq F_T$ be the collection of forbidden edges of $T$ which are incident to at least one edge of $O_T$. Lastly, define $\mathcal{F}(T) = (V(T), E(T) \setminus F_T')$. 
\end{definition}

\begin{definition}
Let $T$ be a tree. We say the vertex $v \in V(T)$ is \textbf{cover-frozen} if $v$ is contained in every minimum cover of $T$ or if $v$ is contained in no minimum cover of $T$. If the vertex $v$ is not cover-frozen, then we call $v$ \textbf{cover-thawed}, i.e., if $v$ is in some, but not all minimum covers of $T$. 
\end{definition}

We recall the following basic fact about matchings and covers.  We refer the interested reader to \cite{LoPl09} for a great deal more on this subject.

\begin{prop}\label{prop:CcoversM}
Let $T$ be a tree, $M\subseteq E(T)$ a maximum matching of $T$ and $C\subseteq V(T)$ a minimum vertex cover of $T$. Then every edge of $M$ contains exactly one vertex of $C$. Furthermore, for every $v\in C$, there exists $e\in M$ so that $v \in e$. 
\end{prop}
\begin{proof}
Since $T$ is a tree, $|M| = |C|$, by the K\H{o}nig-Egerv\'{a}ry Theorem. Since $C$ is a cover of $T$ and $M\subseteq E(T)$, clearly every edge of $M$ contains at least one vertex of $C$. On the other hand, $M$ is an independent collection of edges, so no two elements of $M$ contain the same vertex of $C$. Thus, $|M| = |C|$. 

As for the second part of the claim, $|M| = |C|$ together with the result given by the previous paragraph give the proof. 
\end{proof}

\begin{prop}\label{Claim:incidenttoMT}
Let $T$ be a tree. Then no edge of $O_T$ is incident to an edge of $M_T$. Moreover, if $v\in V(T)$ then the edges incident to $v$ consist of one of the following. 
\begin{itemize}
\item Exactly one edge incident to $v$ is in $M_T$ and the rest are in $F_T$. 
\item Some edges incident to $v$ are in $O_T$ while all others are in $F_T$. 
\end{itemize}
\end{prop}
\begin{proof}
We first justify that no edge of $O_T$ is incident to an edge of $M_T$. If $e \in M_T$, then all other edges incident to $e$ are elements of $F_T$, since the degree of any vertex in a matching is no more than one. 

As for the enumeration in the second part of the claim, if $v\in V(T)$ is incident to an edge of $M_T$, then clearly $v$ is incident to just one edge of $M_T$. The first part of the proof shows that all other edges incident to $v$ come from $F_T$. If $v$ is not incident to an edge of $M_T$, then the edges incident to $v$ fall into one of the other two categories. 
\end{proof}

By the preceding proposition, the components of $\mathcal{F}(T)$ come in exactly two types: either all edges belong to $O_T$, or all edges belong to $M_T \cup F_T$.  The former matching-thawed components are refered to as bc-trees below (see Theorem \ref{Thm:Equivbctree}), while the latter matching-frozen components have a perfect matching.

\begin{prop}\label{Claim:maxmatchingiffrestricts}
Let $T$ be a tree. Then $M$ is a maximum matching of $T$ if and only if $M$ restricts to a maximum matching of the components of $\mathcal{F}(T)$. Additionally, $C$ is a minimum cover of $T$ if and only if $C$ restricts to a minimum cover of the components of $\mathcal{F}(T)$. 
\end{prop}
\begin{proof}
It was already noted that if $M$ is a maximum matching and $C$ is a minimum cover, then $|M| = |C|$. Therefore, the two statements in question are equivalent. We elect to show the matching form of the result. 

Note that by Proposition \ref{Claim:incidenttoMT}, the components of $\mathcal{F}(T)$ take two forms, trees whose edges come from $F_T \cup M_T$, i.e., trees with a perfect matching, and trees with all edges in $O_T$. 

$(\Rightarrow):$ Clearly the claim holds for the first kind of component. Let $X$ be a component of $\mathcal{F}(T)$ with $E(X) \subseteq O_T$. If $v\in V(X)$ so that $\deg_X(v) < \deg_T(v)$, the definition of $\mathcal{F}(T)$ gives that edges of $E(T)\setminus E(X)$ incident to $v$ are forbidden. Thus, the maximum matching $M$ of $T$ restricts to a maximum matching $M_X = M\cap E(X)$ of $X$, as otherwise would contradict that $M$ was a maximum matching of $T$. 

$(\Leftarrow)$: Let $M$ be the union of a maximum matching of each component of $\mathcal{F}(T)$. Clearly, $M$ is a matching of $T$. Suppose $M'$ is a maximum matching of $T$ with $|M| < |M'|$. Since $E(T) \setminus E(\mathcal{F}(T)) = F_T'$, $M'$ is a matching of $\mathcal{F}(T)$. By the other direction of the proof, $M'$ restricts to a maximum matching of every component of $\mathcal{F}(T)$. Since $|M| < |M'|$, there exists a component $X$ of $\mathcal{F}(T)$ so that $|E(X) \cap M| < |E(X) \cap M'|$, contradicting that $M$ is the union of a maximum matching on all components of $\mathcal{F}(T)$. Thus, $M$ is a maximum matching of $T$. 
\end{proof}

The following theorem appears (stated slightly differently) already in work of Harary and Plummer from 1967.

\begin{theorem}\label{Thm:Equivbctree}
\cite{HarPlu1967} The following statements are equivalent for any tree $T$. 
\begin{enumerate}
\item $E(T) = M_T \cup O_T$ (equivalently: $F_T = \emptyset$)
\item $T$ has a unique minimum cover, the cover is independent, and contains no pendant vertex. 
\item If $v, u_1, \dotsc, u_p$ are the pendant vertices of $T$, then $\dist(v,u_i)$ is even for all $1\leq i\leq p$. 
\item $T$ is a bc-tree, i.e., the distance between any pair of pendant vertices is even. 
\item $T$ is the block-cutpoint-tree of some connected graph $G$. 
\end{enumerate}
\end{theorem}

Some components of $\mathcal{F}(T)$ have edge sets fully contained in $O_T$. Thus, the previous theorem implies such a component is a bc-tree, a fact we use repeatedly below. 

\begin{prop}\label{Claim:bctreematchingflexible}
Let $T$ be a bc-tree with unique minimum cover $C$. Then a vertex $v \in V(T)$ is saturated in every maximum matching of $T$ if and only if $v \in C$.  
\end{prop}
\begin{proof}
Clearly, the backwards implication is given by Proposition \ref{prop:CcoversM}. 

For the forward implication, suppose $v \in V(T)$ is saturated in every maximum matching of $T$. Moreover, by way of contradiction, suppose $v \notin C$. Since $v\notin C$ and $C$ is a cover of $T$, $N(v) \subseteq C$.  

Consider rooting $T$ at $v$. Since $T$ is a bc-tree, all leaf edges of $T$ occur at either an even distance from $v$ or all an odd distance from $v$. Moreover, assuming the root occurs at height zero, $C$ contains all vertices at odd heights. Thus, since $C$ contains no leaf edges, the leaf edges of $T$ rooted at $v$ occur at even heights. Since the height of a vertex in $T$ is the same as the distance from $v$, $v$ is an even distance from every leaf edge of $T$. 

Let $M$ be a maximum matching of $T$, and let $P$ be a maximal alternating path in $T$ with respect to edges in $M$ so that $v$ is a pendant vertex of $P$ and the edge of $P$ incident to $v$ is contained in $M$. Since every vertex at odd height is saturated in $M$, $P$ terminates at a leaf edge, $\ell$, of $T$. Since $\dist(v,\ell)$ is even, the last edge of $P$ does not belong to $M$. Thus, $M' := (M \setminus E(P)) \cup (E(P) \setminus M)$ is a matching of $T$ satisfying $|M| = |M'|$ and $v$ is unsaturated in $M'$, a contradiction. 
\end{proof}

\begin{prop}\label{Claim:FincidentToCoverVertices}
Let $T$ be a tree. If $X$ is a bc-tree component of $\mathcal{F}(T)$ containing vertex $v$ satisfying $\deg_T(v) > \deg_X(v)$, then $v$ is contained in the unique minimum cover of $X$. 
\end{prop}
\begin{proof}
Since $X$ is a bc-tree, $X$ contains a unique minimum cover, $C \subseteq V(X)$, and Proposition \ref{Claim:incidenttoMT} gives that $|E(X)| > 1$, so $|C| > 0$. Suppose $v \in V(X)$ with $\deg_T(v) > \deg_X(v)$. Proposition \ref{Claim:incidenttoMT} gives the existence of edge $vu = e \in E(T)$ with $e \in F_T$. By way of contradiction, suppose $v$ is not in the unique minimum cover of $X$. 

By Proposition \ref{Claim:bctreematchingflexible}, there exists a maximum matching $M_X$ of $X$ which leaves $v$ unsaturated. Moreover, Proposition \ref{Claim:maxmatchingiffrestricts} implies there exists a maximum matching $M_T$ of $T$ so that $M_T \cap E(X) = M_X$. 

Since $e \notin M_T$ and $v$ is unsaturated in $M_T$, $u$ is saturated in $M_T$, as otherwise would imply $M_T \cup \{e\}$ is a matching of $T$. Let $e'$ be the edge of $M_T$ incident to $u$. Then $(M_T\setminus \{e'\}) \cup \{e\}$ is a maximum matching of $T$, contradicting that $e \in F_T$, completing the proof. 
\end{proof}

\begin{lemma}\label{lem:perfectmatchingvertexpartition}
Let $T$ be a tree with a perfect matching. Then $T$ contains two minimum covers, $C_1$ and $C_2$ which partition $V(T)$. 
\end{lemma}
\begin{proof}
Let $M$ be a perfect matching of $T$. Root $T$ at vertex $v$. Let $C_1$ be the vertices of $T$ at odd heights of the tree, and $C_2$ the vertices of $T$ at even heights. Clearly, $C_1$ and $C_2$ are vertex covers of $T$, since every edge of $T$ contains one vertex from each height parity, so $|C_1|, |C_2| \geq |M|$. On the other hand, since every edge of $M$ contains one vertex from each height parity class, no edge of $M$ has both vertices in either cover, so $|C_1|, |C_2| \leq |M|$. Therefore, $|C_1| = |C_2| = |M|$, so both covers are minimum covers. The observation that $C_1$ and $C_2$ partitions $V(T)$ is clear. 
\end{proof}

\begin{theorem}\label{Thm:Frozenandthawed}
Let $T$ be a tree
\begin{enumerate}
\item $v \in V(T)$ is cover-frozen if and only if $v$ is not incident to any edges of $M_T$. 
\item $v \in V(T)$ is cover-thawed if and only if every edge incident to $v$ is matching-frozen. 
\item $e \in E(T)$ is matching-thawed if and only if $e$ is incident to two cover-frozen vertices, one in every minimum cover and the other omitted from every minimum cover. 
\item $e \in E(T) \setminus F_T'$ is matching-frozen if and only if $e$ is incident to two cover-thawed vertices. 
\end{enumerate}
\end{theorem}
\begin{proof}
We start with the proof of (1). 

$(\Rightarrow)$: If $v \in V(T)$ is cover-frozen, then $v$ is contained in a component of $\mathcal{F}(T)$ which is a bc-tree, by the remark after Proposition \ref{Claim:incidenttoMT} combined with Lemma \ref{lem:perfectmatchingvertexpartition}.  Another application of Proposition \ref{Claim:incidenttoMT} completes the argument.

$(\Leftarrow)$: If $v$ is not incident to any edges of $M_T$, then again, $v$ is contained in a component of $\mathcal{F}(T)$ which is a bc-tree. Proposition \ref{Claim:maxmatchingiffrestricts} completes the proof.

Note that (2) is merely the contrapositive of (1) since cover-thawed and cover-frozen vertices partition $V(T)$, and Proposition \ref{Claim:incidenttoMT} implies the edges of $T$ incident to a vertex $v \in V(T)$ are either (1) some in $O_T$ and some in $F_T$, or (2) one in $M_T$ and the rest in $F_T$. Thus, a vertex $v$ being incident to an edge of $M_T$ implies the collection of edges incident to $v$ fall into type (2), meaning every edge incident to $v$ is matching-frozen. \\ 

Now we prove (3). 

$(\Rightarrow)$: If $e \in E(T)$ is matching-thawed, then $e$ is contained in a component of $\mathcal{F}(T)$ which is a bc-tree. Proposition \ref{Claim:maxmatchingiffrestricts} completes the proof, noting that the minimum cover of a bc-tree is independent.

$(\Leftarrow)$: We prove this direction by contraposition. If $e$ is incident to two cover-frozen vertices, then there are two cases. Either $e$ is contained in a component of $\mathcal{F}(T)$ which is a bc-tree, in which case the desired result holds, or $e \in F_T'$. If $e \in F_T'$, clearly $e$ is matching-frozen. By the definition of $F_T'$, at least one endpoint of $e$ is incident to edges of $O_T$. Without loss of generality, suppose $e = xy$ and $x$ is incident to edges of $O_T$. Then $x$ is contained in a component of $\mathcal{F}(T)$ which is a bc-tree, so Proposition \ref{Claim:FincidentToCoverVertices} gives that $x$ is a cover vertex. If $y$ is incident to edges of $O_T$, then $y$ is also a cover vertex. On the other hand, if $y$ is not incident to vertices of $O_T$, then $y$ is contained in a component of $\mathcal{F}(T)$ which has a perfect matching. Lemma \ref{lem:perfectmatchingvertexpartition} and Proposition \ref{Claim:maxmatchingiffrestricts} together imply $y$ is not cover-frozen, providing a contradiction. \\

Now we prove (4). 

$(\Rightarrow)$: Suppose $e \in E(T) \setminus F_T'$ is cover-frozen. Then $e$ is contained in a component of $\mathcal{F}(T)$ which has a perfect matching. Thus, Lemma \ref{lem:perfectmatchingvertexpartition} and Proposition \ref{Claim:maxmatchingiffrestricts} together imply the endpoints of $e$ are cover-thawed. 

$(\Leftarrow)$: Now suppose $e \in E(T)$ is incident to two cover-thawed vertices. By Proposition \ref{Claim:FincidentToCoverVertices}, $e \notin F_T'$, so $e$ is contained in a component of $\mathcal{F}(T)$ which has a perfect matching. Thus, Lemma \ref{lem:perfectmatchingvertexpartition} and Proposition \ref{Claim:maxmatchingiffrestricts} together imply the endpoints of $e$ are cover-thawed. 
\end{proof}

\begin{figure}[h]
\centering
\includegraphics[width=4in]{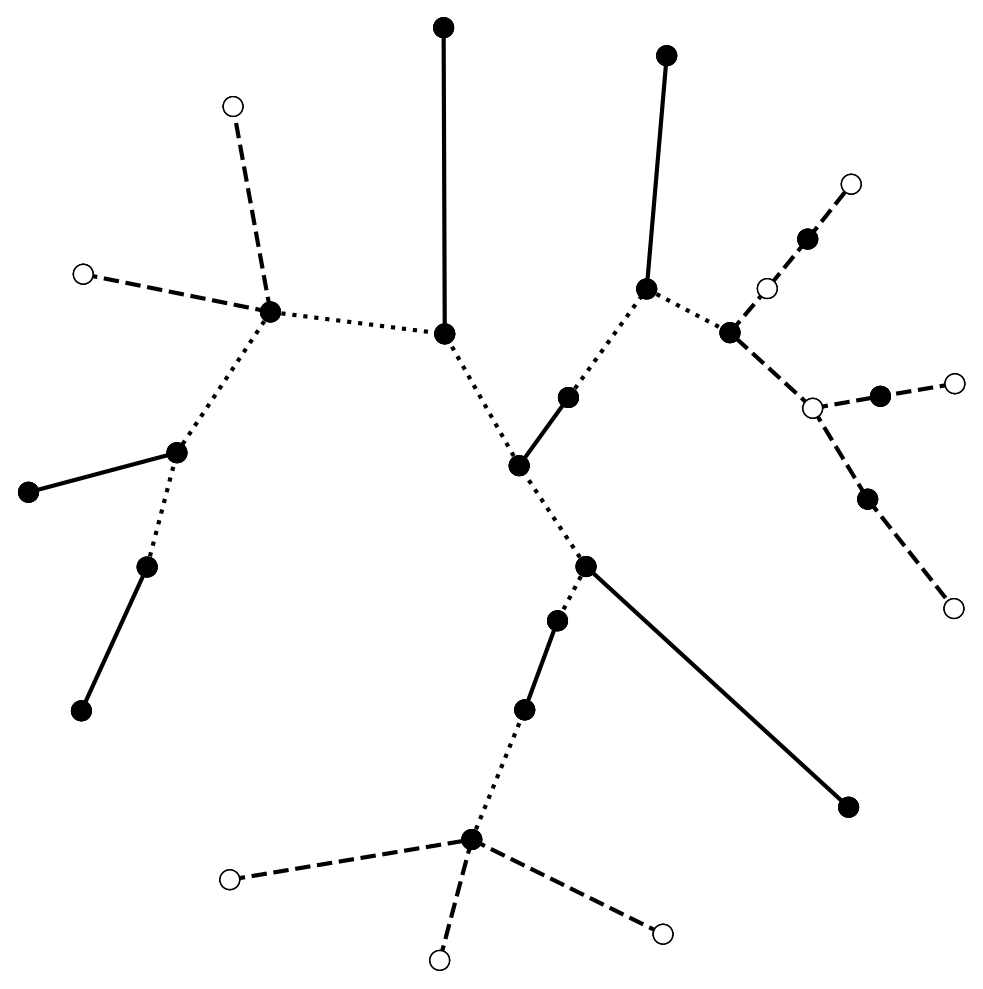}
\caption{A uniformly random tree $T$ on 30 vertices with thermal decomposition: solid = mandatory ($M_T$), dashed = optional ($O_T$), dotted = forbidden ($F_T$).  Filled vertices are the generating set.}
\label{fig:fig1}
\end{figure}

\subsection{Skew Zero Forcing} \label{sec:SZFintro}

\begin{definition}\label{Def:SkewZeroForcing}
\cite{IMAISU2010} Let $G = (V , E)$ be a graph.
\begin{itemize}
\item A subset $S \subseteq V$ defines an initial coloring by filling all vertices of $S$ and leaving all the vertices
not in $S$ unfilled.
\item The skew zero forcing rule says: If a vertex $v \in V$ has exactly one unfilled neighbor, $w$, change the
color of $w$ to filled. In this case we say that $v$ forces $w$.
\item The skew derived set of an initial filled set $S$ is the result of applying the skew zero forcing rule
until no more changes are possible. 
\end{itemize}
\end{definition}

Given a graph $G$ and $S \subseteq V(G)$, define a \textit{skew zero forcing closure} of $S$ (SZF-closure of $S$) to be a skew derived set of the initial coloring $S$. 

\begin{prop}\label{Prop:SkewZeroWellDefined}
Let $G$ be a graph and $S\subseteq V(G)$. Then if $S_1$ and $S_2$ are skew derived sets of $S$, then $S_1 = S_2$. 
\end{prop}
\begin{proof}
Suppose $S_1$ and $S_2$ are two closures of $S$ so that $x \in S_1$. Since $S \subseteq S_1\cap S_2$ it suffices to assume $x \notin S$. 

If there exists a vertex $v \in V(G)$ so that $x \in N(v)$ and $N(v) \setminus S = \{x\}$, then $x \in S_2$, as no sequence of the skew zero forcing rule can increase the number of unfilled neighbors of $v$. 

If this is not the case, there exist sequences of vertices $\{v_i\}_{i=1}^t, \{x_i\}_{i=1}^t \subseteq V(G)$ so that $v_i$ forces $x_i$ for $1 \leq i\leq t$ and then $v_{t+1}:=v$ forces $x_{t+1} := x$ in $S_1$. By the previous paragraph, $x_1 \in S_2$, since $v_1$ has the property that $x_1 \in N(v_1)$ and $N(v_1) \setminus S = \{x_1\}$. Let $S^1 = S\cup \{x_1\}$, so that $S^1 \subseteq S_2$. Recursively define $S^i = S^{i-1} \cup \{x_i\}$ for $2 \leq i \leq n+1$. Suppose also that there exists $k \in \mathbb{Z}$ so that $1 \leq k \leq n$ and $S^k \subseteq S_2$. Then $x_{k+1} \in S_2$, since the construction of $S_1$ gives that $v_{k+1}$ has the property that $x_{k+1} \in N(v_{k+1})$ and $N(v_{k+1}) \setminus S^k = \{x_{k+1}\}$, so that $x_{k+1}$ is eventually added to $S_2$ if it is not already an element of $S^k \subseteq S_2$. 
\end{proof}

Thus, {\it the} skew zero forcing closure of a set $S \subseteq V(G)$ is well-defined.  We denote it by $\overline{S}$.  If $S = \overline{S}$, then we say that $S$ is skew zero forcing {\it closed}, which some literature refers to as ``stalled''.  This is a meaningful term for any closure operator; see the beginning of Section \ref{Sec:Matroids} for more.  We sometimes refer to any set $S \subseteq V(G)$ so that $\overline{S} = V(G)$ as a ``skew zero forcing set''.

\begin{prop}\label{Prop:SkewZeroClosure}
Let $G$ be a graph. If $A, B\subseteq V(G)$ so that $A\subseteq B$, then $\overline{A} \subseteq \overline{B}$. 
\end{prop}
\begin{proof}
Let $v_1,\dotsc,v_t \in V(G)$ be the collection of vertices, in order of inclusion via repeated application of the skew zero forcing rule, in $\overline{A} \setminus A$. Further, let $u_1, \dotsc, u_t \in V(G)$ so that the skew zero forcing rule applied to $u_i$ resulted in the inclusion of vertex $v_i$. By induction, we show that $\{v_i\}_{i=1}^t \subseteq \overline{B}$. 

Let $j \in \mathbb{Z}$ so that $1 \leq j\leq t$. If $j = 1$, then all neighbors of $u_1$ except $v_1$ are contained in $A$. Since $A\subseteq B$, all neighbors of $u_1$ except $v_1$ are in $B$. Regardless of whether $v_1$ is contained in $B$ or not, it is clear that $A \cup \{v_i\}_{i=1}^1 \subseteq \overline{B}$. Now suppose there exists $k \in \mathbb{Z}$ so that $k\geq 1$ and $A \cup \{v_i\}_{i=1}^k \subseteq \overline{B}$. We consider $v_{k+1}$. By the definition of the skew zero forcing rule, all neighbors of $u_{k+1}$ except $v_{k+1}$ are contained in $A \cup \{v_i\}_{i=1}^k$. By the induction hypothesis, $A \cup \{v_i\}_{i=1}^k \subseteq \overline{B}$, so all neighbors of $u_{k+1}$ except (possibly) $v_{k+1}$ are contained in $\overline{B}$. Thus, $v_{k+1} \in \overline{B}$, completing the proof of the desired inclusion. 
\end{proof}

\begin{prop}\label{Prop:SkewZeroClosedGeneratenullvectors}
Let $T$ be a tree. The set $S \subseteq V(T)$ is skew zero forcing closed if and only if there exists $\mathbf{x} \in \ker(T)$ so that the zeros of $\mathbf{x}$ occur exactly at the vertices of $S$. 
\end{prop}
\begin{proof}
$(\Leftarrow)$: This implication is clear. 
$(\Rightarrow)$: Let $S \subseteq V(T)$ be a skew zero forcing closed set. Consider rooting $T$ at a leaf edge $\ell$. We construct a nullvector $\mathbf{x}$ with entries $x_v$ for $v \in V(T)$ by iteratively working through the tree $T$. Start by assigning zeros to coordinates corresponding to vertices of $S$, so let $x_v = 0$ for each $v \in S$. Now we choose values for all nonzero coordinates of $\mathbf{x}$. 

If $\ell \notin S$, let $x_\ell = 1$. Furthermore, let $v_1$ be the unique neighbor of $\ell$ (note that $v_1$ is the only vertex of $T$ at height one). Since $S$ is skew zero forcing closed, $v_1 \in S$, so $x_{v_1} = 0$. 

Let $v \in V(T)$ be at height $h \geq 2$, so that $v$ is not a pendant vertex, and if $u \in V(T)$ is at height less than $h$, $x_u$ has already been assigned. Let $w$ be the unique neighbor of $v$ at height $h-1$. If $x_w = 0$, then either all neighbors of $v$ are in $S$ or at least two are not. If the former is the case, $x_u = 0$ for all neighbors $u$ of $v$ at height $h+1$. If the latter is the case, let $N(v) \setminus S = \{u_1, u_2, \dotsc, u_m\}$. Then, define $x_{u_1} = 1$ and $x_{u_i} = -1/(m-1)$ for $2 \leq i\leq m$. In either case, we have that $\sum_{w : vw \in E(T)} x_w = 0$. 

The other case considers if $x_w\neq 0$. In this case, we note that $x_w$ has already been assigned. Let $N(v) \setminus S = \{u_1, u_2, \dotsc, u_m\}$ and define $x_{u_i} = -x_w / m$. Thus $\sum_{w : vw \in E(T)} x_w = 0$. 

Note that using this recursive algorithm, it is possible to populate all entries of $\mathbf{x}$. Moreover, the above argument shows $\sum_{w : vw \in E(T)} x_w = 0$ for all vertices $v$ of $T$ except pendant vertices. Since pendant vertices have a unique neighbor and $S$ is skew zero forcing closed, we have that $\sum_{w : vw \in E(T)} x_w = 0$ for all pendant vertices $v \in V(T)$ as well, completing the proof. 
\end{proof}

Note that the backwards implication above holds for any graph, not just trees: if $S \subseteq V(G)$ is the zero locus of some nullvector $\mathbf{v}$, then $S$ is SZF-closed.  Indeed, if any vertex $x$ had exactly one neighbor $y$ for which $v_y \neq 0$, i.e., $y \not \in S$, then the $x$ coordinate of $A(G) \mathbf{v}$ would also be nonzero, a contradiction.  This is made more precise in Proposition \ref{Prop:matroidContainment}.

\subsection{Dulmage-Mendelsohn Decomposition and Characterizing Generating Sets}

\begin{definition}\label{Def:DMdecomp}
\cite{DulMen1958, Pul1995} Let $G$ be a bipartite graph, $M$ a maximum-cardinality matching in $G$, and $V_0$ the set of vertices of $G$ unsaturated by $M$ (the ``free vertices''). Then $G$ can be partitioned into three parts
\begin{itemize}
\item $E$ - the vertices reachable from $V_0$ by an $M$-alternating path of even length.
\item $O$ - the vertices reachable from $V_0$ by an $M$-alternating path of odd length.
\item $U$ - the vertices unreachable from $V_0$ by an $M$-alternating path.
\end{itemize}
\end{definition}
It is well-known that the Dulmage-Mendelsohn decomposition is a special case of the Gallai-Edmonds decomposition, and these decompositions have highly useful properties when considering maximum matchings and minimum covers of graphs. 

\begin{theorem}[\cite{LoPl09}]
\label{Thm:DMdecomp}
If $M$ and $N$ are two maximum-matchings of bipartite graph $G$, then $M$ and $N$ define the same $(U, E, O)$ decomposition. 
\end{theorem}

For any graph $G$, it will be convenient to define the subset $\mathcal{V}(S)$ of $\ker(G)$ {\it generated} by a set $S \subseteq V(G)$.  This definition is extended in Section \ref{Sec:Hypergraphs} to hypergraphs.

\begin{definition}
    If $S \subseteq V(G)$, then we denote by $\mathcal{V}^G(S)$ the subspace $\{\mathbf{v} \in \ker(G) : x \in S \Rightarrow v_x = 0\}$.  If $X \subseteq \ker(G)$, then we say that $S$ ``generates'' $X$ if $\mathcal{V}^G(S) = X$.
\end{definition}

Note that, given a set $X \subseteq \ker(G)$, if there is a set $S \subseteq V(G)$ which generates $X$, then there is a maximal set which does so: $\{x : (\forall \mathbf{v} \in X) (v_x = 0)\}$, the intersection of all zero loci of vectors in $X$, which we refer to as ``the'' generating set of $X$.

\begin{theorem}\label{Thm:BigEquiv}
Let $T$ be a tree, and $S \subseteq V(T)$. Then the following are equivalent. 
\begin{enumerate}
\item $\mathcal{V}^{T}(S) = \ker(T)$, i.e., $S$ is the generating set for the nullspace of $T$, i.e., $S$ is the set of common indices of zeros of all nullvectors. 
\item $S$ is the skew zero forcing closure of $\emptyset$.
\item $S$ is the union of all vertices in matching-frozen components of $\mathcal{F}(T)$ and the unique minimum cover of matching-thawed components of $\mathcal{F}(T)$. 
\item $S$ is the union of all minimum covers of $T$. 
\item $S$ is the intersection of all sets of saturated vertices in maximum matchings of $T$. 
\item $S = U\cup O$ in the Dulmage-Mendelsohn Decomposition of $T$.
\end{enumerate}
\end{theorem}

\begin{proof}
$2 \Leftrightarrow 5$: By Corollary \ref{Cor:Gammoid} (which does not use results from outside subsection \ref{Sec:MatchingsandMatroids}; see discussion preceding the proof), the sets unsaturated by maximum matchings are precisely the minimal sets $S \subseteq V(T)$ so that $\overline{S} = V(T)$.  Therefore, if $M$ is a maximum matching of $T$, and $X_M$ is the collection of vertices unsaturated by $M$, then the minimality of the skew zero forcing set $X_M$ gives $X_M \cap \overline{\emptyset} = \emptyset$, since $\overline{X_M \setminus \overline{\emptyset}} \supseteq \overline{(X_M \setminus \overline{\emptyset}) \cup \overline{\emptyset}} = V(T)$ by Proposition \ref{Prop:SkewZeroClosure}. Thus, $\overline{\emptyset} \subseteq V(M)$ for every maximum matching $M$, so if $S \subseteq V(T)$ is the collection of vertices saturated by every maximum matching, $\overline{\emptyset} \subseteq S$.  Now, suppose $x \not \in \overline{\emptyset}$, and let $M$ be a maximum matching.  If there is a maximum matching $M$ which does not saturate $x$, then $x \not \in S$.  Suppose $x_0 := x$ is saturated in $M$; we construct a maximal $M$-alternating path $P = x_0 x_1 x_2\ldots x_m$ of even length starting with $x$ and the edge $x x_1 \in M$ so that $x_{2t} \not \in \overline{\emptyset}$ for each $t \geq 0$.  We claim $x_m$ is unmatched: otherwise, $m$ is odd, but then it has only one neighbor -- namely, $x_{m-1}$ -- outside the set $\overline{\emptyset}$, contradicting that $\overline{\emptyset}$ is SZF-closed.  Thus, $M' = (M \setminus E(P)) \cup (E(P) \setminus M)$ is a maximal matching which does not saturate $x$, so $x \not \in S$.  We may conclude that $S \subseteq \overline{\emptyset}$.

$1 \Rightarrow 2$: Since $S$ is the collection of common zeros of all nullvectors of $T$, $S$ is clearly skew zero forcing closed by the remark following Proposition \ref{Prop:SkewZeroClosedGeneratenullvectors}. Since $\emptyset\subseteq S$, $\overline{\emptyset} \subseteq \overline{S} = S$. It remains to show $S \subseteq \overline{\emptyset}$. Since $S$ is the generating set for $\ker(T)$, $v_s = 0$ for every $\mathbf{v} \in \ker(T)$ and every $s \in S$. By definition, $\mathcal{V}^T(\overline{\emptyset}) \subseteq \ker(T)$, so this inclusion implies that $v_s = 0$ for every $\mathbf{v} \in \mathcal{V}^T(\overline{\emptyset})$ and every $s \in S$. Thus, $S \subseteq \overline{\emptyset}$.

$2 \Rightarrow 1$: Let $S = \overline{\emptyset}$, and let $G$ be the generating set for the nullspace of $T$. By Proposition \ref{Prop:SkewZeroClosedGeneratenullvectors}, $\overline{\emptyset}$ is the zero locus of some nullvector of $T$. Thus, since $G$ is the collection of all coordinates which are zero for every nullvector of $T$, $G\subseteq \overline{\emptyset}$. 
 
On the other hand, Proposition \ref{Prop:SkewZeroClosedGeneratenullvectors} implies that $G$ is skew zero forcing closed, meaning $G = \overline{G}$. Thus, Proposition \ref{Prop:SkewZeroClosure} and $\emptyset \subseteq G$ give $S = \overline{\emptyset} \subseteq \overline{G} = G$, completing the proof of $S = G$. 

$3 \Leftrightarrow 4$: Let $S$ be the union of all vertices in matching-frozen components of $\mathcal{F}(T)$ and the unique minimum cover of matching-thawed components of $\mathcal{F}(T)$. Recall from the comment after Proposition \ref{Claim:incidenttoMT} that components of $\mathcal{F}(T)$ are each either a bc-tree or have a perfect matching. Moreover, by Theorem \ref{Thm:Frozenandthawed}, the matching-frozen components of $\mathcal{F}(T)$ are those components with a perfect matching, since every vertex of a perfectly matched component of $\mathcal{F}(T)$ is incident only to matching-frozen edges. Similarly, Theorem \ref{Thm:Frozenandthawed} implies that the matching-thawed components of $\mathcal{F}(T)$ are the bc-tree components, since every vertex of a bc-tree component of $\mathcal{F}(T)$ is not incident to any edge of $M_T$. Now, let $\mathcal{C}$ be the collection of all minimum covers of $T$. By Proposition \ref{Claim:maxmatchingiffrestricts}, any cover in $\mathcal{C}$ restricts to a minimum cover of each component of $\mathcal{F}(T)$. Let $X$ be a component of $\mathcal{F}(T)$. 

If $X$ is a bc-tree, then Theorem \ref{Thm:Equivbctree} implies $X$ has a unique minimum cover. Proposition \ref{Claim:maxmatchingiffrestricts} further implies $C\cap V(X)$ is the unique minimum cover in $X$ for any $C\in \mathcal{C}$. Thus, $S\cap V(X) = (\bigcup \mathcal{C}) \cap V(X)$ in this case for the bc-tree/matching-thawed components of $\mathcal{F}(T)$.

On the other hand, if $X$ contains a perfect matching, Lemma \ref{lem:perfectmatchingvertexpartition} gives the existence of two covers $C_1, C_2 \in \mathcal{C}$ so that $(C_1 \cap V(X)) \cup (C_2 \cap V(X)) = V(X)$. Thus, $S\cap V(X) = (\cup \mathcal{C}) \cap V(X)$ in this case for the perfect matching/matching-frozen components of $\mathcal{F}(T)$.

$4 \Leftrightarrow 5$: Let $\mathcal{C}$ be the collection of all minimum covers of $T$, and define $S = \bigcup \mathcal{C}$. 
%Suppose $S$ is the union of all minimum covers, and let $\mathcal{C}$ be the collection of all minimum covers of $T$. 
By Proposition \ref{prop:CcoversM}, for every $C \in \mathcal{C}$ and every $v \in C$, $v$ is saturated by every maximum matching of $T$. Thus, $S$ is contained in the intersection of sets of vertices saturated by maximum matchings of $T$. 

As for the other inclusion, suppose $v \in V(T)$ so that $v$ is saturated by every maximum matching of $T$. Let $X$ be the component of $\mathcal{F}(T)$ containing $v$. If $X$ contains a perfect matching, then Lemma \ref{lem:perfectmatchingvertexpartition} and Proposition \ref{Claim:maxmatchingiffrestricts} give that $V(X) \subseteq S$, so $v \in S$. On the other hand, if $X$ is a bc-tree, then Proposition \ref{Claim:bctreematchingflexible} gives that $v$ is contained in the minimum cover of $X$, and Proposition \ref{Claim:maxmatchingiffrestricts} gives that $v \in S$, completing the proof of the desired equality. 

$5 \Leftrightarrow 6$: (This is easily deduced from Theorem 3.2.1 in \cite{LoPl09}; we include a proof here for completeness.) Suppose that $S$ is the intersection of sets of saturated vertices in maximum matchings of $T$. As is given by previous proofs, $S$ contains every vertex of perfect matching components of $\mathcal{F}(T)$ and exactly the unique minimum cover of bc-tree components of $\mathcal{F}(T)$. 

Let $(U, E, O)$ be the Dulmage-Mendelsohn decomposition of $T$. Let $v \in V(T)$, and let $X$ be the component of $\mathcal{F}(T)$ which contains $v$. If $X$ is a bc-tree, then Proposition \ref{Claim:bctreematchingflexible} gives that $v$ is not an element of the unique minimum cover of $X$ if and only if $v$ is unsaturated in some maximum matching of $T$. Thus, if $C$ is the unique minimum cover of $X$, then $V(X) \setminus C \subseteq E$, because $E$ contains all unsaturated vertices. Moreover, since $C$ is independent, for every $c \in C$, $N_X(c) \subseteq V(X) \setminus C \subseteq E$. Thus, $C \subseteq O$. 

On the other hand, if $X$ contains a perfect matching, then all vertices of $X$ are saturated by every maximum matching of $T$. Let $M$ be such a matching, and $V_0$ be the vertices of $T$ unsaturated by $M$. 
By way of contradiction, suppose $v \notin U$, meaning there exists path $P$ with terminal vertices $v$ and $u \in V(T)$ so that $P$ is alternating with respect to $M$ and $u \in V_0$. Let $Y$ be the component of $\mathcal{F}(T)$ containing $u$. Since $u \in V_0$, $Y$ is a bc-tree, and $u$ is not contained in the minimum cover of $Y$. Let this minimum cover be $C_Y$. Since $Y$ and $X$ are different components of $\mathcal{F}(T)$, there exists edge $e \in F_T' \cap E(P)$ so that $\{u_1\} := e\cap V(Y)$. By Proposition \ref{Claim:FincidentToCoverVertices}, $u_1 \in C_Y$, meaning $u_1 \neq u$. Thus, if $P'$ denotes the alternating (with respect to $M$) subpath of $P$ with endpoints $u_1$ and $u$, $|E(P')|> 0$. Since $P$ is an alternating path and $e \notin M$, the edge of $P'$ incident to $u_1$ is in $M$. Since $u \in V(P')$ is unmatched in $M$, the edge of $P'$ incident to $u$ is not in $M$, so $P'$ an alternating path implies $\dist(u_1,u)$ is even. However, $u_1 \in C_Y$ and Theorem \ref{Thm:Equivbctree} imply $C_Y$ is exactly the vertices of $Y$ an even distance from $u_1$. Thus, $u \in C_Y$, a contradiction, and so $v \in U$, whence $V(X) \subseteq U$ and we may conclude that $S \subseteq U\cup O$.

%%% HERE

To show $U \cup O\subseteq S$, we show instead $S' \subseteq E$, where we define $S' = V(T) \setminus S$. In a bc-tree component $X$ of $\mathcal{F}(T)$, we know $S' \cap V(X)$ is exactly the vertices not in the minimum cover. By Proposition \ref{Claim:FincidentToCoverVertices}, each of these vertices are omitted from some maximum matching of $T$. Thus, Theorem \ref{Thm:DMdecomp} implies $S' \cap V(X) \subseteq E$. For a perfect matching component $X$ in $\mathcal{F}(T)$, $S' \cap V(X) = \emptyset$, so the inclusion $S' \cap V(X) \subseteq E$ is trivial.

\end{proof}

The equivalence of (1) and (5) appears essentially as Corollary 19 of \cite{SaSa09}.

\subsection{Edge Influence on Rank in the Thermal Decomposition}

In this section, we describe another equivalent formulation of the thermal decomposition which is convenient for computation.  Let $\rank(G)$ denote $\rank(A(G))$ for any graph $G$. Let $\nu(G)$ denote the size of a maximum matching in $G$. Let $\eta(G)$ denote the nullity of any graph $G$. Finally, let $c(G)$ denote the size of a minimum cover of $G$. The authors of \cite{CveGut1972} prove that for a tree $T$ on $n \geq 1$ vertices with maximum matching size $\nu(T)$, the nullity of $T$ is given by $\eta(T) = n - 2\nu(T)$. Thus, it is also the case that $\rank(T) = 2\nu(T)$. 

\begin{prop}\label{Prop:rankformandatory}
Let $T$ be a tree with $e \in E(T)$. Then $e \in M_T$ if and only if $\rank(T) = \rank(T-e) + 2$. 
\end{prop}
\begin{proof}
Note that if $\nu(T')$ denotes the size of the maximum matching of the tree $T'$, then $\nu(T') = \rank(T') / 2$. 

Suppose first that $\rank(T) = \rank(T-e) + 2$. Then, $\nu(T) = \nu(T - e) + 1$. Thus, every maximum matching of $T$ contains $e$, giving $e \in M_T$. Now suppose that $e \in M_T$. Then, every maximum matching of $T$ contains $e$, so $\nu(T) = \nu(T - e) + 1$. By the relationship between $\nu(X)$ and $\rank(X)$ for any tree $X$, we have the desired result. 
\end{proof}

\begin{prop}
Let $T$ be a tree, and let $e \in E(T)$. If $\rank(T) = \rank(T-e) + 2$, then $\rank(T-e) = \rank(T/e)$. 
\end{prop}
\begin{proof}
Since $\rank(T) = \rank(T-e) + 2$, the proof of Proposition \ref{Prop:rankformandatory} gives that every maximum matching of $T$ contains $e$. Clearly, there are two possibilities for the relationship between $\nu(T)$ and $\nu(T/e)$. Either $\nu(T) = \nu(T/e)$ or $\nu(T) = \nu(T/e) + 1$. If the former is true, then let $M$ be a maximum matching of $T/e$. Then $M$ is also a maximum matching of $T$ which omits $e$, a contradiction. Thus, $\nu(T) = \nu(T/e) + 1$, implying $\rank(T) = \rank(T/e) + 2$.
\end{proof}

\begin{prop}
Let $T$ be a tree with $e \in E(T)$, so that $\rank(T) = \rank(T-e)$. Then $e \in O_T$ if and only if $\rank(T) = \rank(T/e)$. 
\end{prop}
\begin{proof}
First suppose $\rank(T) = \rank(T/ e)$. Since $\rank(T) = \rank(T-e)$, $\nu(T) = \nu(T-e)$. Thus, there exists a maximum matching of $T$ which does not contain $e$, so $e \notin M_T$. So, at least one endpoint of $e$ is saturated by every maximum matching of $T$. On the other hand, since $\rank(T) = \rank(T/ e)$, there exists a maximum matching of $T$ which leaves one endpoint of $e$ unsaturated, since we can simply take a maximum matching of $T/e$ and extend it to a maximum matching of $T$. Let $M$ be this matching, $e = uv$, and suppose $u$ is the unique endpoint of $e$ saturated in $M$. Let $e'$ be the matching edge incident to $u$. Then $(M \setminus \{e'\}) \cup \{e\}$ is another maximum matching of $T$, proving $e \in O_T$, finishing the proof in this direction.  

Now, suppose $\rank(T) = \rank(T/ e) + 2$, and let $e = uv$. Note that $\rank(T) = \rank(T/ e) + 2$ implies $\nu(T) = \nu(T/ e) + 1$ and $c(T) = c(T/e) + 1$. Let the contracted vertex in $T/e$ be $v'$.  Suppose $M$ is a maximum matching of $T$ that contains $e$. Let $M' = M \setminus \{e\}$, so that $M'$ is a matching of $T/e$. Note that $\nu(T) = \nu(T/ e) + 1$ implies $M'$ is a maximum matching of $T/e$. Since $v'$ is unsaturated by $M'$, Proposition \ref{prop:CcoversM} gives that every minimum cover of $T/e$ omits $v'$, meaning all neighbors of $v'$ in $T/e$ are contained in every minimum cover of $T/e$. Let $C'$ be such a minimum cover. Then $C'$ is also a cover of $T-e$, giving that $c(T/e) \geq c(T-e)$, but the assumption $\rank(T) = \rank(T-e)$ implies $c(T) = c(T-e)$. Substitution gives $c(T/e) \geq c(T)$, but $c(T) = c(T/e) + 1$, a contradiction. Therefore, no maximum matching of $T$ containing $e$ exists, proving $e \in F_T$. 
\end{proof}

From the previous propositions, it is clear that
$e \in F_T$ if and only if $\rank(T) = \rank(T-e) = \rank(T\setminus e) + 2$. Therefore, we summarize these results as follows.
\begin{theorem} For any tree $T$ and edge $e \in E(G)$,
\begin{enumerate}
\item $e \in M_T$ iff $\rank(T) = \rank(T - e) + 2$ (this also implies $\rank(T) = \rank(T/e) + 2$).
\item $e \in O_T$ iff $\rank(T) = \rank(T/e)$ (this also implies $\rank(T) = \rank(T-e)$). 
\item $e \in F_T$ iff $\rank(T) = \rank(T-e)$ and $\rank(T) = \rank(T/e) + 2$.
\end{enumerate}
\end{theorem}

\section{Kernel and Skew Zero Forcing Matroids} \label{Sec:Matroids}

In the following section, we relate the set system of zero loci of nullvectors and SZF-closed sets of a graph.  We begin by showing the former is a linear matroid, the ``kernel matroid''.  In the next subsection, we show that, when the SZF-closure is a matroid closure operator, there is a polynomial time algorithm for computing the SZF number of a graph, despite the fact that it is NP-hard in general.  We then show in the next subsection that this property of the SZF-closed sets holds for trees and graphs which are ``SZF-complete'', i.e., for which the zero loci of nullvectors and the SZF-closed sets are the same.  We show that trees are SZF-complete; in general, the complement of any skew zero forcing set is saturated by some matching; and that, for some special bipartite graphs (including trees), they are precisely the complements of matching-saturated sets, and are therefore ``gammoids''.   In the last subsection, we show that many classes of graphs are SZF-complete, and we characterize the nonsingular bipartite SZF-complete graphs.

First, a useful definition:

\begin{definition}\label{Def:Matroid}
A set $E$ together with a map $\overline{\,\cdot\,} : \mathscr{P} (E) \rightarrow \mathscr{P}(E)$ defines a closure operator on $E$ if the following are satisfied. 
\begin{enumerate}
\item (Extensive Property) $X\subseteq \overline{X}$ for each $X \in \mathscr{P}(E)$.
\item (Idempotent Property) $\overline{X} = \overline{\overline{X}}$ for each $X \in \mathscr{P}(E)$. 
\item (Monotone Property) $\overline{X} \subseteq \overline{Y}$ for any $X, Y\in \mathscr{P}(E)$ with $X \subseteq Y$. 
\end{enumerate}

\noindent If, in addition, the closure operator satisfies the following, it is a matroid closure operator.

\begin{enumerate}
    \item[4.] (Mac Lane-Steinitz Exchange Property) For all elements $a, b$ of $E$ and all subsets $X$ of $E$, if $a \in \overline{X\cup \{b\}} \setminus \overline{X}$, then $b \in \overline{X\cup \{a\}} \setminus \overline{X}$. 
\end{enumerate}

\end{definition}

Whole libraries have been written about the sublime world of matroids, among which \cite{We95} is a valuable resource.  Here we simply remark that, given a matroid closure operator $\overline{\,\cdot\,}$: closed sets or ``flats'' are those $S \subseteq E$ so that $S = \overline{S}$, which form a lattice under set inclusion; bases are minimal sets so that $\overline{S} = E$, which always have the same cardinality; independent sets are subsets of bases; dependent sets are non-independent sets; circuits are minimal dependent sets; and hyperplanes are maximal proper closed sets.  Furthermore, the ``rank'' function which assigns to $S \subseteq E$ the size of the smallest set $F \subseteq S$ so that $\overline{F} \supseteq S$ is also the rank function of the lattice of flats as a poset.

Note that the results of subsection \ref{sec:SZFintro} imply that the SZF-closure is a bona fide closure operator.  It is not always a matroid closure, however, so we investigate when it is below in subsection \ref{sec:szfmatroids}.

\subsection{The Kernel Matroid}

Given a graph $G$, if $S\subseteq V(G)$, then we call $S$ \textit{realizable} if there exists $\mathbf{x} \in \ker(G)$ so that $S = Z(\mathbf{x})$. We examine the structure of the collection of realizable subsets of $V(G)$.

\begin{prop}\label{Prop:RealizableClosedUnderFiniteIntersection}
Let $G$ be a graph, and $\mathcal{X} = \{X_i\}_{i=1}^t$ a collection of realizable subsets of $V(G)$. If $S = \cap_{i\in [t]} X_i$, then $S$ is realizable. 
\end{prop}
\begin{proof}
Clearly the $t = 1$ case is trivial, so we show the result for $t = 2$, from which the general case follows by induction. Since $X_1$ and $X_2$ are realizable, let $\mathbf{x},\mathbf{y}\in \ker(G)$ so that $X_1 = Z(\mathbf{x})$ and $X_2 = Z(\mathbf{y})$. Further, let $Q = \{-x_v/y_v : v \in V(G), y_v\neq 0\}$. Take any $r \in \mathbb{R}^* \setminus Q$. Then, $r\mathbf{y} + \mathbf{x} \in \ker(G)$, and $Z(r\mathbf{y}+\mathbf{x}) = S$, as quick examination shows $ry_v + x_v = 0$ if and only if $y_v = x_v = 0$ for any $v\in V(G)$. 
\end{proof}

Write $\mathbf{e}_j$ to denote the elementary vector whose support is $\{j\}$. For graph $G$, let $\im(G)$ denote the image of the adjacency matrix $A(G)$. Furthermore, if $X\subseteq V(G)$,  define $\spn(X) \subseteq \mathbb{R}^{V(G)}$ to be the vector space spanned by elementary vectors corresponding to vertices of $X$, i.e., $\spn(X) = \spn(\{\mathbf{e}_v : v \in X\})$. 

\begin{definition}\label{Def:Widehat}
Let $G$ be a graph and $S \subseteq V(G)$. Define $\widehat{S} := \{v : \mathbf{e}_v \in \im(G) + \spn(S)\}$. 
\end{definition}

\begin{prop}\label{Prop:OrthogonalComplementWidehatDefined}
Let $G$ be a graph and $S\subseteq V(G)$. Then $\widehat{S}$ is the intersection of all realizable sets containing $S$.
\end{prop}
\begin{proof}
Let $\mathcal{X}$ be the collection of all realizable sets containing $S$, let $Y = \bigcap \mathcal{X}$, and suppose $X \in\mathcal{X}$. By definition, $X \subseteq \widehat{X} = \{x : \mathbf{e}_x \in \im(G) + \spn(X)\}$. Since $S\subseteq X$, $\spn(S) \subseteq \spn(X)$, so $\im(G) + \spn(S) \subseteq \im(G) + \spn(X)$, giving $\widehat{S} \subseteq Y$. It remains to show the reverse inclusion. Since $\mathcal{X}$ is a finite collection, Proposition \ref{Prop:RealizableClosedUnderFiniteIntersection} gives $Y$ is a realizable set. By definition, $\widehat{S}$ is a realizable set containing $S$, so $Y\subseteq \widehat{S}$, completing the proof. 
\end{proof}

\begin{cor}\label{Cor:AlternateWidehatDef}
$\widehat{S}$ is the minimum realizable set containing $S$ with respect to inclusion. Further, $S$ is realizable if and only if $\widehat{S} = S$. 
\end{cor}

\begin{prop}\label{Prop:WidehatMatroid}
Let $G$ be a graph and $\mathscr{P}(V(G))$ the collection of %realizable 
subsets of $V(G)$. Then the map which sends $v \mapsto \mathbf{e}_v + \im(G)$ is an isomorphism between the matroid $\widehat{\,\cdot\,}$ on $\mathscr{P}(V(G))$ and the linear matroid given by the collection of vectors $\{\mathbf{e_v} + \im(G)\}_{v \in V(G)}$. 
\end{prop}
\begin{proof}
Let $f : V(G) \to \mathbb{R}^{V(G)} / \im(G)$ be the coset map, i.e., $f(v) = \mathbf{e}_v + \im(G)$.  Let $S \subseteq V(G)$.  
Suppose $v \in \widehat{S}$.  By Definition \ref{Def:Widehat}, $\mathbf{e}_v \in \im(G) + \spn(S)$, so $f(v) \in \spn(f(S))$.  Conversely, if $\mathbf{w} \in \spn(f(S)) \cap f(V(G))$, then $\mathbf{w} = \mathbf{e}_u + \im(G)$ for some $u \in V(G)$, so $\mathbf{w} = \mathbf{e}_u + \im(G) \in \spn(f(S)) = \spn(S) + \im(G)$ implies $u \in \widehat{S}$ and $\mathbf{w} \in f(\widehat{S})$.  Then $\widehat{S} = f^{-1}\left (\spn(f(S)) \cap f(V(G)) \right )
$, from which the result follows.
\end{proof}

\subsection{Skew Zero Forcing Matroids} \label{sec:szfmatroids}

The authors of \cite{AnsJacPenSaa2016} asked for the computational complexity of computing various forcing numbers, mentioning that it is known  that the decision problems of bounding by $k$ the zero forcing numbers (\cite{Aaz2008}) and positive semidefinite zero forcing numbers (\cite{FalMeaYan2016}) are NP-hard. Since then, \cite{Shi2017} established that the failed zero-forcing number and failed skew zero forcing number are NP-hard to bound as well.  We add to this by showing the skew zero forcing number is NP-hard to bound in general, although by contrast, there is a polynomial time algorithm to compute skew zero forcing numbers for graphs whose SZF-closure gives rise to a matroid.  Then, below, we show that there are many graphs whose SZF-closure is a matroid closure: trees, cycles with length divisible by $4$, bipartite graphs which have a unique perfect matching, complete bipartite graphs, graphs derived from these by appending a path of length $2$ or subdividing an edge into a path of length $5$, and bipartite graphs in which no maximum matching admits an alternating cycle.

In the next few results, we refer to ``ordinary'' zero forcing, where the rule is that a vertex $v$ which belongs to the set $X \subseteq V(G)$ can force the addition of its neighbor $w$ to $X$ if $w$ is the only unfilled neighbor of $v$.  Note that this rule contrasts with the skew zero forcing rule in that it requires that $v$ be filled for it to force.

\begin{prop}
  Let $G$ be a graph on $n$ vertices, with zero forcing number $z(G)$. Define $G'$ to be a graph with vertex set $V(G) \times \{1,2,3\}$ and edges $E(G') = A \cup B$, where $A = \{\{(v,1),(w,1)\}: vw \in E(G)\}$ and $B = \{\{(v,i),(v,j)\}: 1 \leq i < j \leq 3 \}$. If $X \subseteq V(G')$ is minimal such that $\overline{X} = V(G')$, then $|X| =  z(G)$.
\end{prop}
\begin{proof}
  Denote by $\Delta_v$ the triangle induced by $\{v\} \times [3]$ for each $v \in V(G)$, and write $X = X_0,X_1,\ldots,X_t = \overline{X}$ for the results of a sequence of one-vertex SZF rule applications.  If there exists a $v$ so that $\Delta_v \subseteq X$, let $X' = X \setminus \{(v,1)\}$.  Then taking $X'_0 = X'$ and $X'_r = X_{r-1}$ for each $1 \leq r \leq t+1$ yields a SZF rule application sequence, so $\overline{X'} = V(G')$, contradicting the minimality of $X$.  Furthermore, if $(v,j) \in X_r$ for some $r$ and $j \in [3]$, then $\Delta_v \subseteq \overline{X_r}$.  Therefore, $\Delta_v \cap X$ is either empty or contains exactly one vertex, in which case we assume without loss of generality that $(v,1) \in X$.

  Note that, if $(v,1) \not \in X_r$ for some $r$, then $\Delta_v \cap X_k = \emptyset$ for each $k \leq r$.  In fact, no SZF rule can be applied at $(v,1)$ in $X_k$ for $k \leq r$, since $(v,1)$ has at least two unfilled neighbors.  Thus, any $r$ so that a SZF rule is applied at $(v,1)$ must have $(v,1) \in X_r$.  In other words, the sequence $X_r \cap (V(G) \times \{1\})$ is (other than some steps when no changes occur) an ordinary zero forcing rule application sequence applied to $V(G) \times \{1\}$, and this sequence projects onto $V(G)$ as an ordinary zero forcing rule application sequence there.  So $\proj_{V(G)} X$ is a zero forcing set for $G$.

  Conversely, if $X$ is a zero forcing set for $G$, it is easy to see that $X \times \{1\}$ is a skew zero forcing set for $G'$.
\end{proof}

\begin{cor}
    The decision problem, ``Is the SZF number of $G$ less than or equal to $k$?'', is NP-hard.
\end{cor}
\begin{proof}
    This follows from the above reduction and Aazami's result (\cite{Aaz2008}) that the ordinary zero forcing number decision problem is NP-hard.
\end{proof}

\begin{theorem} If SZF-closure in $G$ is a matroid closure operator, then there is an $O(n^3)$ algorithm to compute its skew zero forcing number.
\end{theorem}
\begin{proof} The rank of a matroid is the size of the smallest set whose closure is the whole ground set, so the rank of the SZF matroid of $G$ is the cardinality of any minimal set whose SZF-closure contains all of $V(G)$, i.e., the skew zero forcing number of $G$.  A form of the result then follows from \cite{RobWel1980}, since they show that an oracle capable of computing closures can be used to compute rank in polynomial time.  To be concrete, we provide the algorithm here.  Identify the vertex set $V(G)$ with $[n]$, where $n = |V(G)|$, for convenience.\\

\begin{algorithmic}
\State $S \gets \emptyset$
\State $k \gets 0$
\While{$\overline{S} \neq V(G)$} 
    \State $x \gets \min(V(G) \setminus S)$
    \State $k \gets k+1$
\EndWhile
\State \textbf{return} $k$\Comment{Return the SZF-number of $G$}\\
\end{algorithmic}

This algorithm succeeds because the vertices counted by $k$ are a basis.  Note that the iteration executes at most $n$ times.  Furthermore, a subroutine is needed to compute $\overline{S}$ from $S$; this can be done in time $O(n^2)$ by the following algorithm:\\

\begin{algorithmic}
\State $(\forall x \in V(G)) (A_x \gets N(x) \setminus S)$
\While{$(\exists x)(|A_x|=1)$}
    \State{$S \gets S \cup A_x$}
    \State{$X \gets A_x$}
    \For{$z \in V(G)$}
        \State{$A_z \gets A_z \setminus X$}
    \EndFor
\EndWhile
\State \textbf{return} $S$\Comment{Return the SZF-closure $\overline{S}$ of $S$}\\
\end{algorithmic}

The set $A_x$ keeps track of the vertices which are neighbors of $x$ but have not yet joined $\overline{S}$.  This subroutine takes $O(n^2)$ time because (1) the time complexity of computing $N(x) \setminus S$ is $O(n)$ for each of $n$ vertices; and (2) the inner loop executes $n$ times.  In total, we have an $O(n^3)$ execution time.
\end{proof}

Note that the above algorithm always returns the cardinality of {\it some} set whose closure is $G$, even if the SZF-closed sets are not the flats of a matroid.  Therefore, it provides an upper bound in general.

\subsection{Matchings and the Skew Zero Forcing Matroid}\label{Sec:MatchingsandMatroids}

\begin{definition}\label{Def:sfcomplete}
Call the graph $G$ \textbf{skew zero forcing complete (SZF-complete)} if every SZF-closed $S \subseteq V(G)$ is the zero locus of a nullvector of $G$, i.e., $G$ is SZF-complete if and only if the families of sets closed under $\widehat{\,\cdot\,}$ and $\overline{\,\cdot\,}$, respectively, are the same. 
\end{definition}

By Proposition \ref{Prop:WidehatMatroid}, if a $G$ is SZF-complete, then SZF-closure also gives rise to a matroid, which we term the {\it SZF matroid} of $G$.  Note that Proposition \ref{Prop:SkewZeroClosedGeneratenullvectors} says that trees are SZF-complete.  Thus, we ask: which other graphs are SZF-complete?

\begin{prop}\label{Prop:matroidContainment}
Let $G$ be a graph, and $X \subseteq V(G)$. If $X = \widehat{X}$, then $X = \overline{X}$. 
\end{prop}
\begin{proof}
If $X = \widehat{X}$, then $X$ is realizable by Corollary \ref{Cor:AlternateWidehatDef}. Suppose $X \neq \overline{X}$, meaning there exists $v \in V(G)$ and $u \not \in X$ so that $v$ forces $u$ under the SZF rule. If $\mathbf{x} \in\ker(G)$ so that $Z(\mathbf{x})= X$, then $(A\mathbf{x})_v = \sum_{w \sim v} x_w = x_u = 0$, so $u \in X$, a contradiction. 
\end{proof}

The following results serve to identify the bases of the skew zero forcing and kernel matroids for trees. The authors of \cite{GutBor2011} extend the work of \cite{CveGut1972} by showing that if $G$ is a $C_{4s}$-free bipartite graph, then $\eta(G) = |V(G)| - 2 \nu(G)$. 

\begin{prop}\label{Prop:C4sDulMenNullity}
Let $G$ be a $C_{4s}$-free bipartite graph, and $(U, E, O)$ the Dulmage-Mendelsohn decomposition of $G$. Then $\eta(G) = |E| - |O|$. 
\end{prop}
\begin{proof}
By the work of \cite{GutBor2011}, $\eta(G) = |V(G)| - 2\nu(G)$. The following computation completes the proof, noting that $|V(G)| = |E| + |O| + |U|$, and $\mu(G) = |O| + |U|/2$:
\begin{align*}
\eta(G) &= |V(G)| - 2\nu(G) \\
&= (|E| + |O| + |U|) - (2|O| + |U|) \\
&= |E| - |O|.
\end{align*}
\end{proof}

In particular, Proposition \ref{Prop:C4sDulMenNullity} gives the size of any basis for the SZF/kernel matroid of a tree.  The next few results explore which vertex sets with size $|E| - |O|$ actually form a basis for any $G$.

\begin{prop} \label{prop:matching} Let $G$ be a graph, $X \subseteq V(G)$ and $\overline{X}$ its SZF-closure.  Then there is a matching $M$ of $G$ in which all vertices of $\overline{X} \setminus X$ are saturated.
\end{prop}
\begin{proof} Let $x_1,\ldots,x_m$ be the sequence of vertices added to $X = X_0,X_1,\ldots,X_m=\overline{X}$ as the skew zero forcing rule is applied, where $m = |\overline{X} \setminus X|$.  Note that $x_j$ is added only because there was a vertex $y_j \in V(G)$ whose only neighbor outside $X_j$ is $x_j$.  A vertex $y_j$ never recurs in such a sequence, since, once its single unfilled neighbor is filled, it never has any unfilled neighbors again.  No $x_j$ occurs twice in the sequences of $x_i$'s, since it only gets filled once.  Furthermore, $x_jy_j \in E(G)$ for each $1 \leq j \leq m$, since $y_j$ is a neighbor of $x_j$.  Therefore, the set $D$ of directed edges $\{(y_j,x_j)\}_{j=1}^m$ is an oriented subgraph of $G$ with out-degrees and in-degrees at most one. Some components of $D$ are cycles and others are paths; furthermore, $D$ spans $\overline{X}$.  

We claim that any such cycle has length at most $2$.  Suppose the vertex set of a cycle of length $t$ in $D$ is (in order) $v_0, \ldots, v_{t-1}$.  Note that each $v_i$ is a $y_j$ for some $j$ and an $x_k$ for some $k$. Without loss of generality, $v_1 = x_j$ is the first vertex of $D$ added by SZF rule application.  Then, if $v_{t-1} = x_k$, then $k > j$ and so $x_k = v_{t-1} \not \in X_j$.  But then $v_{t-1}$ and $v_1$ are neighbors of $v_0 = y_j$ not contained in the set $X_j$, contradicting the fact that $v_1 = x_j$ was the only unfilled neighbor of $v_0$ when $x_j$ was added to $X_j$ to obtain $X_{j+1}$ unless $v_1 = v_{t-1}$, i.e., $t=2$.  Thus, $D$ consists of $2$-cycles and paths.  From each such cycle, add the corresponding undirected edge to a set $M$; for each path, add alternating edges to $M$, starting with the sink vertex.  Since the source $y$ of such a path has in-degree zero, it is not $x_j$ for any $j$, so $y \not \in \overline{X}$.  Thus, $M$ is a matching of all vertices of $\overline{X}$ except a subset of $S \subseteq X$.
\end{proof}

%What is this Corollary here?  I don't see what it's doing for us.
%\begin{cor} If the graph $G$ is SZF-complete and $\mathbf{v}$ is a nullvector of $G$ with zero locus $X$ such that $\widehat{X} = V(G)$, then $G$ admits a matching which saturates all vertices of $V(G) \setminus X$. 
%\end{cor}
%\begin{proof}
%    Since $G$ is SZF-complete, $\overline{X} = \widehat{X} = V(G)$ implies $G$ admits a matching which saturates $V(G) \setminus X$ by Proposition \ref{prop:matching}.
%\end{proof}

The following result is stated with only one direction of proof in the unpublished manuscript \cite{DeA2014}.  An important consequence is that, for trees $T$, minimal sets $X \subseteq V(T)$ so that $\overline{X}=V(T)$ are precisely the sets of vertices unsaturated by some maximum matching.  Note that the proof only invokes Corollary \ref{prop:matching}, which itself does not use any consequences of Theorem \ref{Thm:BigEquiv}.

\begin{cor}\label{Cor:Gammoid}
For any bipartite graph $G$ in which no maximum matching admits an alternating cycle, the SZF matroid is dual to the matching matroid of $G$.  That is, the SZF matroid is a ``gammoid.''
\end{cor}
\begin{proof}
    DeAlba (\cite{DeA2014} Proposition 4.1) showed that the unsaturated vertices of every maximum matching are a minimal skew forcing set, i.e., a basis of the SZF matroid.  The proof proceeds as follows: Let $M$ be a maximum matching, of cardinality $r=|M|$, and let $X = S \cup S'$ be the set of its saturated vertices bipartitioned according to the bipartition of $G$. For a vertex $v \in V(M)$, write $v'$ for the vertex to which it is matched. Note that $V(G) \setminus X$ are filled. Since $G[X]$ contains no cycles, there is a vertex $v_1 \in S$ of degree one in this subgraph, so $v_1$ forces $v_1' \in S'$.  Then $G[X \setminus \{v_1,v_1'\}]$ is acyclic, so there is another vertex $v_2 \in S$ which forces $v_2' \in S'$, and so on until $v_{r}$ forces $v_{r}'$.  For each $j \in [r]$, $v_j$ has no edges in $G$ to $v_i'$ if $i>j$.  Now, $v_r$ is the only unfilled neighbor of $v_r'$, so the former is forced by latter; then, $v_{r-1}$ is the only unfilled neighbor of $v_{r-1}'$, so $v_{r-1}$ is forced; and so on, until finally all of $X$, and therefore $V(G)$, is filled.
    
    Conversely, if $B$ is a basis of the SZF matroid, then $\overline{X} = V(G)$, so Proposition \ref{prop:matching} implies that $G$ admits a matching $M$ which saturates $V(G) \setminus X$.  Then $M$ can always be turned into maximum matching $M'$ so that $\bigcup M \subseteq \bigcup M'$ by applying augmenting paths, per Berge's Lemma.
\end{proof}

\subsection{SZF-Completeness for Classes of Graphs}\label{Sec:SZFComplete}

Here we describe some explicit classes of graphs which are SZF-complete.

\begin{prop} \label{prop:onlyfourven}
The cycle $C_n$ is SZF-complete if and only if $4 | n$. 
\end{prop}
\begin{proof}
For convenience, suppose the vertices of $C_n$ are labeled by $[n]$ in order around the cycle. 

Suppose $4|n$. By Proposition \ref{Prop:matroidContainment}, it suffices to show the SZF-closed sets are realizable. Quick examination shows $C_n$ has four SZF-closed sets, namely $\emptyset$, $V(C_n)$, $\{x\in [n] : 2 |x\}$, and $\{x\in [n] : 2 \nmid x\}$. For $\emptyset$, take the corresponding vector to have $j$ coordinate $(-1)^j$. For $V(C_n)$, the corresponding vector is the trivial $\mathbf{0}$. For the remaining two SZF-closed sets, the corresponding vector places zeros at vertices of the specified set and alternates $1$ and $-1$ at vertices outside the specified set. It is straightforward to verify that these are nullvectors.

Now suppose $4 \nmid n$. It is well known that the eigenvalues of the cycle $C_n$ are $2\cos(2\pi j/n)$ for $0 \leq j \leq n-1$. Thus, $0$ is an eigenvalue of a cycle if and only if its length is divisible by four. Therefore, zero is not an eigenvalue of $C_n$, but $\emptyset$ is an SZF-closed set, completing the proof. 
\end{proof}

\begin{prop}
The complete bipartite graph $K_{m,n}$ is SZF-complete. 
\end{prop}
\begin{proof}
Let the partition classes of $K_{m,n}$ be $M$ and $N$ so that $|M| = m$ and $|N| = n$. By Proposition \ref{Prop:matroidContainment}, it suffices to show an arbitrary SZF-closed set is realizable. Suppose $X \subseteq V(K_{m,n})$ is an SZF-closed set so that $X \cap M = R$ and $X \cap N = B$. By definition, $|M| - |R| \neq 1$ and $|N| - |B| \neq 1$. If $|M| - |R| > 1$ (resp. $|N| - |B| > 1$), then define $R' := M\setminus R$ (resp. $B' := N\setminus B$). In the corresponding vector, assign $1$ to the coordinate corresponding to one arbitrarily chosen vertex of $R'$ (resp. $B'$), and the remaining vertices of $R'$ (resp. $B'$) with $-1 / (|R'| - 1)$ (resp. $-1 / (|B'| - 1)$). All other coordinates (i.e., the ones corresponding to vertices of $X$) are assigned $0$.  It is easy to see that this is a nullvector realizing the set $X$. 
\end{proof}

Appending a path of length $2$ to an SZF-complete graph yields another SZF-complete graph.

\begin{prop}
Suppose $G$ is SZF-complete and $x\in V(G)$. If $G' = (V(G) \cup \{y,z\}, E(G) \cup \{xy,yz\})$, where $y,z\notin V(G)$, then $G'$ is also SZF-complete.  
\end{prop}
\begin{proof}
By Proposition \ref{Prop:matroidContainment} it suffices to show an arbitrary SZF-closed set is realizable. 

Suppose $S' \subseteq V(G')$ is an SZF-closed set. Note that $y \in S'$, since otherwise $z$ would have exactly one non-$S'$ neighbor.  Furthermore, $z \in S'$ if and only if $x \in S'$, since otherwise $y$ would have exactly one non-$S'$ neighbor.  Let $S = V(G) \cap S'$.  The number of non-$S$ neighbors of every vertex of $V(G)$, including $x$, is equal to its number of non-$S'$ neighbors in $G'$, so $S$ is SZF-closed in $G$.  Since $G$ is SZF-complete, there is a nullvector $\mathbf{v}$ with zeros given by $S$.  Let $\mathbf{v}'$ be the vector in $\mathbb{R}^{G'}$ given by $\mathbf{v} \oplus (0,-\mathbf{v}_x)$.

Then it is easy to check that the zero locus of $\mathbf{v}'$ is exactly $S'$, and that $\mathbf{v}'$ is a nullvector for $G'$, completing the proof. 
\end{proof}

A graph $G$ \textbf{uniquely perfect matchable (UPM)} if $G$ has exactly one perfect matching. Godsil showed (\cite{God1985}) that bipartite UPM graphs are exactly the perfectly-matchable subgraphs of the half-graph.  We will employ the following additional characterization from Theorem 2 of \cite{WanShaYua2015}:

\begin{lemma} \label{lem:UPM} A bipartite graph $G$ with bipartition $(X, Y )$ is a UPM-graph if and only if
\begin{enumerate}
    \item each of $X$ and $Y$ contains a pendant vertex, and
    \item when the pendant vertices and their neighbors are deleted, the resulting subgraph has a unique perfect matching.
\end{enumerate}
\end{lemma}

The following result adds a property to the TFAE statement Theorem 2.12 of \cite{AnsJacPenSaa2016}, and characterizes SZF-completeness for nonsingular bipartite graphs.

\begin{theorem}
A nonsingular bipartite graph $G$ is SZF-complete if and only if it is UPM. 
\end{theorem}
\begin{proof}
$(\Rightarrow)$: We proceed by strong induction on $|V(G)|$.  It is easy to check the statement holds for 1 or 2 vertices.  Suppose $G$ is bipartite, SZF-complete, and nonsingular on at least 3 vertices. Since $G$ is nonsingular, $V(G)$ is the only zero locus of a nullvector of $G$. Moreover, since $G$ is SZF-complete, $\emptyset$ is not SZF-closed. So, there must be a vertex $v$ of degree one, with neighbor $w$.  Then (see \cite{AkbKir2007}, Lemma 1), $|\det(G - \{v,w\})| = |\det(G)|$, so $G'=G-\{v,w\}$ is nonsingular.  Take a set $S'$ which is SZF-closed for $G'$.  If $w$ has at least one neighbor in $V(G')-S'$, let $S = S' \cup \{w\}$.  Then $S$ is SZF-closed for $G$, so there is a nullvector $\mathbf{x}$ whose zero coordinates are exactly $S$.  Restrict $\mathbf{x}$ to $G'$ to obtain $\mathbf{x}'$, which is a nullvector for $G'$ corresponding to $S'$.  If $w$ has no $S'$-neighbors, let $S = S' \cup \{v,w\}$.  Then $S$ is SZF-closed for $G$, so there is a nullvector $\mathbf{x}$ whose zero coordinates are exactly $S$. Restrict $\mathbf{x}$ to $G'$ to obtain $\mathbf{x}'$, which is a nullvector for $G'$ corresponding to $S'$.  Therefore, each SZF-closed set for $G'$ corresponds to a nullvector for $G'$, and $G'$ is SZF-complete as well.  Now, since $G$ is nonsingular, the permanent of the adjacency matrix of $G$ is nonzero; since this permanent is the square of the number of perfect matchings in $G$ (consider its adjacency matrix written in block form), $G$ must have a perfect matching.  Every such perfect matching must include the edge $vw$.  Now $G'$ is bipartite, nonsingular, and SZF-complete, so $G'$ also has a unique perfect matching $M$ by the induction hypothesis, and $M \cup \{vw\}$ is the unique perfect matching of $G$.

$(\Leftarrow)$: Again, $G$ is UPM implies that $|\det(G)|$ is 1, so $G$ is nonsingular.  We show SZF-completeness by induction.  The base case is easy.  By Lemma \ref{lem:UPM}, $G$ has a pendant vertex $v$ with neighbor $w$ so that $G' = G-\{v,w\}$ is also bipartite UPM and therefore nonsingular.  Let $S$ be an SZF-closed set of vertices for $G$.  Then $S$ contains $w$, or else $v$ would have exactly one non-$S$ neighbor.  Let $S'=S-\{v,w\}$ in $G'$.  Then $S'$ is SZF-closed in $G$ because the number of non-$S'$ neighbors of vertices $u$ in $V(G')$ is the same as the number of non-$S$ neighbors of $u$ in $G$.  Thus, by the induction hypothesis, there is a nullvector $\mathbf{x'}$ for $G'$ whose zero coordinates are exactly $S'$.  Extend $\mathbf{x'}$ to a vector $\mathbf{x}$ by setting the $w$ coordinate to 0 and the $v$ coordinate equal to the negative of the sum $C$ of the $\mathbf{x'}$ coordinates arising from neighbors of $w$.  Then $\mathbf{x}$ is a nullvector for $G$, and its set of zeros is exactly $S$ unless $S$ does not contain $v$, but the $v$-coordinate $-C$ of $\mathbf{x}$ is zero.  Then the nullspace of $G$ has a nontrivial intersection with the space where the $v$ and $w$ coordinates are equal, contradicting the nonsingularity of $G$, unless $S'$ is all of $V(G')$.  But then $S=V(G)$, and the zero vector corresponds to $S$.  So $G$ is SZF-complete.
\end{proof}

Subdividing an edge of an SZF-complete graph into a path of length $5$ yields another SZF-complete graph.

\begin{prop}
%just say that the edge is being subdividing into a $P_6$ instead of the difficult notation
Let $G$ be an SZF-complete graph, and $xy \in E(G)$. If $G' = (V(G) \cup \{x_i\}_{i=1}^4, [E(G) \setminus \{xy\}] \cup \{xx_1,x_1x_2,x_2x_3,x_3x_4,x_4y\})$, then $G'$ is SZF-complete.
\end{prop}
\begin{proof}
Let $X'\subseteq V(G')$ be SZF-closed, and let $X = X'\cap 
V(G)$. Notice that $x \in X'$ if and only if $x_2,x_4 \in X'$. Moreover, $y \in X'$ if and only if $x_1,x_3 \in X'$. We show $X$ is SZF-closed in $G$. 

Clearly, $N_G(v) = N_{G'}(v)$ for each $v \in V(G) \setminus \{x,y\}$. Additionally, $N_G(x) = (N_{G'}(x) \setminus \{x_1\})\cup \{y\}$, and $N_G(y) = (N_{G'}(y) \setminus \{x_4\})\cup \{x\}$. Since $x \in X'$ if and only if $x_4 \in X'$, and $y \in X'$ if and only if $x_1 \in X'$, we have that $X$ is SZF-closed in $G$. %(is more detail necessary?). Thus, $X$ is SZF-closed in $G$. 

Since $G$ is SZF-complete, there exists $\mathbf{z} \in \ker(G)$ so that $X$ is the zero locus of $\mathbf{z}$. Extend $\mathbf{z} \in \mathbb{R}^{V(G)}$ to vector $\mathbf{z}' \in \mathbb{R}^{V(G')}$ so that $z_v = z'_v$ for each $v \in V(G)$, $z'_{x_2} := -z_x$, $z'_{x_4} := z_x$, $z'_{x_1} := z_y$, and $z'_{x_3} := -z_y$. Clearly, $\mathbf{z}' \in \ker(G')$, so $X'$ is realizable, and Proposition \ref{Prop:matroidContainment} completes the proof. 
\end{proof}

Note that SZF-completeness is {\it not} required for a graph's SZF-closed set system to arise from a matroid.  For example, the SZF-closed sets of $C_6$ with vertex set $[6]$ are $\emptyset$, $\{1,3,5\}$, $\{2,4,6\}$, and $[6]$, which form a $2$-dimensional Boolean lattice, but $C_6$ is not SZF-complete by Proposition \ref{prop:onlyfourven}.

\section{Hypergraphs} \label{Sec:Hypergraphs}

The present work began with an investigation into adjacency nullvectors of hypergraphs, so we return to that topic here. As is often the case, hypergraphs add significant additional complexities to the situation for ordinary graphs.  At least for linear hypertrees -- connected hypergraphs with no nontrivial cycles and for which pairs of edges intersect in at most one vertex -- we extend some of Theorem \ref{Thm:BigEquiv} in the first subsection below.  This involves a new definition of SZF-closed sets for hypergraphs, which we show gives rise to a set system in containment-preserving bijection with the lattice of subvarieties of linear hypertrees' nullvarieties (terminology defined below).  The following section then describes the kernel-closed and SZF-closed sets of complete hypergraphs.

Spectral hypergraph theory is a large and growing area, so it is not possible to offer a thorough introduction here.  We provide only key definitions.  For a starting point on hypergraph spectra, we refer the reader to \cite{CoDu12}; for more on eigenvarieties, see \cite{FanBaoHua19}; for a broader view from the theory of tensors, see \cite{QiLu17}.  A {\it hypergraph} $\mathcal{H}$ is a pair $(V(\mathcal{H}), E(\mathcal{H}))$ of vertices and edges, with $E(\mathcal{H}) \subseteq \mathscr{P}(V(\mathcal{H}))$, where we assume that $|e| > 1$ for each $e \in E(\mathcal{H})$; we denote the number of vertices of $\mathcal{H}$ by $n = n_\mathcal{H}$ and typically identify $V(\mathcal{H})$ with $[n]$.  The {\it rank} of an edge $e \in E(\mathcal{H})$ is its cardinality, and $\mathcal{H}$ is said to be $k$-uniform, or a $k$-graph, if all edges have rank $k$. The {\it adjacency hypermatrix} $\mathcal{A}_\mathcal{H}$ of a $k$-uniform hypergraph $\mathcal{H}$ on $n$ vertices is a dimension-$n$, order-$k$ hypermatrix (often identified with the tensor of which it is the coordinate matrix), i.e., an element of $\mathbb{C}^{[n]^k}$, whose $(i_1,\ldots,i_k)$ entry $\mathcal{A}_\mathcal{H}(i_1,\ldots,i_k)$ is $1/(k-1)!$ if $\{i_1,\ldots,i_k\}$ is an edge of $\cH$ and zero otherwise.  The factor of $1/(k-1)!$ is sometimes omitted in this definition -- for present purposes, it is immaterial.  The value $\lambda \in \mathbb{C}$ is an eigenvalue with eigenvector $\mathbf{x} = (x_1,\ldots,x_n) \in \mathbb{C}^n$ if 
$$
\sum_{i_2,\ldots,i_k} \mathcal{A}_\mathcal{H}(i,i_2,\ldots,i_k) x_{i_2} \cdots x_{i_k} = \lambda x_i^{k-1} 
$$
for each $i \in [n]$.  The $k$-form $p(x_1,\ldots,x_k)=\sum_{i_1,i_2,\ldots,i_k} \mathcal{A}_\mathcal{H}(i_1,\ldots,i_k) x_{i_1} \cdots x_{i_k}$ whose gradient appears on the left-hand side above is sometimes known as the \textit{Lagrangian polynomial} of $\cH$.  It is straightforward to show (see \cite{Qi05} proof of Proposition 1) that the coordinate $\partial p/\partial x_i$ is $k$ times the Lagrangian polynomial of the \textit{link} of vertex $v_i$ in $\cH$, i.e., the hypergraph whose edges are $\{e \setminus \{v_i\} \,|\, v_i \in e \in E(\cH)\}$. So, for $k$-uniform hypergraphs, it is equivalent to define an eigenvalue $\lambda$ with eigenvector $\mathbf{x}$ as a simultaneous solution to 
\begin{equation} \label{eq:eq1}
\sum_{\substack{e\in E(\mathcal{H})\\ v\in e }} \prod_{\substack{w\in e \\ w \neq v}} x_w = \lambda x_v^{k-1}
\end{equation}
for all $v \in V(\mathcal{H})$. For ease of notation, we use $f_{\mathcal{H}, v}$ to denote the polynomial on the left-hand side of (\ref{eq:eq1}), the Lagrangian polynomial of the link of $v$ in $\mathcal{H}$. Even for mixed rank hypergraphs, i.e., non-uniform hypergraphs, we adopt this definition for eigenpairs $(\lambda,\mathbf{x})$. Notice that for graphs, i.e., when $k = 2$, this agrees with standard adjacency spectra in graph theory. 

Throughout this section, we are interested in the eigenvalue zero, in the same way as for graphs previously. For hypergraph $\mathcal{H}$, the collection of all eigenvectors associated to the eigenvalue zero, called \textit{nullvectors}, form an affine variety, called the \textit{nullvariety}. The authors of \cite{CooFic2021} use $\mathcal{V}_0(\mathcal{H})$ to denote this collection, but since these vectors comprise the kernel of $\mathcal{A}_\mathcal{H}$ as a $(k-1)$-form (at least for $k$-uniform hypergraphs), we denote the collection of nullvectors by $\ker(\mathcal{H})$ here to emphasize its relationship to graphs' nullspaces.

Given the collection of links' polynomials, $f_{\mathcal{H}, v}$, we require a notation for evaluating some variables at zero. Write  $\mathbf{x}_U = \{x_v\}_{v \in U}$ for any set $U \subseteq V(\mathcal{H})$ and $\langle \mathcal{F} \rangle$ for the polynomial ideal in $\mathbb{C}[\mathbf{x}] = \mathbb{C}[\{x_v\}_{v \in V(\mathcal{H})}]$ generated by a collection of polynomials $\mathcal{F}$ over $\{x_v\}_{v \in V(\mathcal{H})}$.  Furthermore, let $\phi_{U} : \mathbb{C}[\mathbf{x}] \rightarrow \mathbb{C}[\mathbf{x}]$ be the evaluation homomorphism obtained by extending the maps $x_v \mapsto 0$ if $v \in U$ and $x_v \mapsto x_v$ otherwise. Lastly, for $U \subseteq V(\mathcal{H})$, define $\mathcal{V}^{\mathcal{H}}(U)$ to be the affine variety defined by the ideal $\langle \mathbf{x}_U \cup \{ f_{\mathcal{H}, v} : v \in V(\mathcal{H})\}\rangle = \langle \mathbf{x}_U \rangle + \langle \phi_{U} f_{\mathcal{H}, v} : v \in V(\mathcal{H})\}\rangle $, which captures what is left of $\langle \{f_{\mathcal{H}, v} : v \in V(\mathcal{H})\} \rangle$ after variables indexed by elements of $U$ are set to zero. Notice that for any $U\subseteq V(\mathcal{H})$, $\mathcal{V}^{\mathcal{H}}(U)$ is a subvariety of $\ker(\mathcal{H})$, and $\ker(\mathcal{H}) = \mathcal{V}^{\mathcal{H}}(\emptyset)$ in particular.

The collection of graph nullvectors form a vector space; analogously, the collection of hypergraph nullvectors forms an algebraic variety. While there is a unique generating set for $\ker(T)$ when $T$ is a tree, for hypertrees $\mathcal{T}$, $\ker(\mathcal{T})$ breaks into many irreducible components, each having its own generating set.  We therefore examine the generating sets of irreducible components of $\ker(\mathcal{T})$.

\subsection{Linear Hypertrees} 

A hypergraph $\mathcal{H}$ is {\it linear} if every pair of edges intersect in at most one vertex.  A {\it cycle} in $\mathcal{H}$ is a sequence $x_0,e_1,x_1,\ldots,x_{t-1},e_t,x_t$ of alternating vertices $x_j \in V(\mathcal{H})$ and edges $e_j \in E(\mathcal{H})$ so that the $x_j$ are distinct except that $x_0 = x_t$, and the $e_j$ are distinct.  A hypergraph $\mathcal{H}$ is a {\it hypertree} if it admits no cycles.  A {\it pendant vertex} is a vertex of degree one, and a {\it leaf edge} is an edge containing at most one non-pendant vertex.

\begin{prop}\label{Prop:vertex cover}
Let $\mathcal{T}$ be a linear hypertree on $n$ vertices so that every edge has rank at least two. Let $\mathbf{x} \in \mathbb{C}^n$ be a nullvector of $\mathcal{T}$. Then for each edge $e\in E(\mathcal{T})$, there exists vertex $v\in e$ so that $x_v = 0$. 
\end{prop}
\begin{proof}
By way of contradiction, suppose that there exists an edge $e$ so that $x_v \neq 0$ for all $v\in e$. Let $E$ be the collection of all such edges of $\mathcal{T}$ with this property, and further define $\mathcal{T}'$ to be the subgraph of $\mathcal{T}$ containing all edges in $E$ (and no isolated vertices). Then $\mathcal{T}'$ is a nonempty forest, so $\mathcal{T}'$ contains a pendant vertex, $v$. Since $f_{\mathcal{T}',v}$ contains as an addend exactly one monomial which evaluates to a product of nonzero values, $f_{\mathcal{T}',v}(\mathbf{x}) \neq 0$. By the construction of $\mathcal{T}'$, $f_{\mathcal{T}',v}(\mathbf{x}) = f_{\mathcal{T},v}(\mathbf{x})$ (since any monomials corresponding to edges in $E(\mathcal{T}) \setminus E$ incident to $v$ have at least one vertex $u$ with $x_u = 0$), further implying $f_{\mathcal{T},v}(\mathbf{x}) \neq 0$, a contradiction. 
\end{proof}

\begin{prop}\label{Prop:noMoreThanTwo}
Let $\mathcal{T}$ be a (not necessarily uniform) linear hypertree with pendant edge $e$. Furthermore, let $S \subseteq V(\mathcal{T})$ so that $\mathcal{V}^{\mathcal{T}}(S)$ is an irreducible component of $\ker(\mathcal{T})$. Then $|e \cap S| \leq 2$. 
%Let $\mathcal{T}$ be a (not necessarily uniform) linear hypertree with pendant edge $e$. Furthermore, let $S$ be the set of polynomials generating an irreducible component of $\ker(\mathcal{T})$. If $\mathbf{x} \in \mathcal{V}^{\mathcal{T}}(S)$, then $|\{x_v : v \in e\} \cap S| \leq 2$. 
\end{prop}
\begin{proof}
By way of contradiction, suppose $|e \cap S| \geq 3$. Let $A = e \cap S$. Then at least two vertices of $A$ are pendant vertices of $e$, namely $v_1$ and $v_2$. Then $\mathcal{V}^\mathcal{T}(S) \subsetneq \mathcal{V}^{\mathcal{T}}(S \setminus \{v_1\})$, but $\mathcal{V}^{\mathcal{T}}(S \setminus \{v_1\})$ is irreducible since $\phi_{S\setminus \{v_1\}} f_{\mathcal{T}, v} = \phi_S f_{\mathcal{T}, v}$ for every $v$ not a pendant vertex of $e$, and $\phi_{S\setminus \{v_1\}} f_{\mathcal{T}, u} = 0$ for every pendant vertex $u \in e$, since $|A\setminus \{v_1\}| \geq 2$. Thus, the irreducibility of $\mathcal{V}^{\mathcal{T}}(S \setminus \{v_1\})$ contradicts that $\mathcal{V}^{\mathcal{T}}(S)$ is an irreducible component of $\ker(\mathcal{T})$, completing the proof. 
%By way of contradiction, suppose $|\{x_v : v \in e\} \cap S| \geq 3$. For ease of discussion, let $A = \{x_v : v \in e\} \cap S$. Then at least two variables of $A$ correspond to pendant vertices of $\mathcal{T}$. Let $x_{v_1}$ and $x_{v_2}$ be such variables. Then the variety generated by $S \setminus \{x_{v_1}\}$ still generates nullvectors of $\mathcal{T}$ and contains the variety generated by $S$, contradicting that $S$ is a generating set for an irreducible component of $\ker(\mathcal{T})$. 
\end{proof}

Notice also that Proposition \ref{Prop:vertex cover} applies to $2$-trees as well as general linear hypertrees: the zero locus of any null vector is a vertex cover. However, in both the tree and hypertree settings, this proposition is not an equivalence, i.e., there are vertex covers of trees/hypertrees which do not correspond to the zero set of a nullvector. Proposition \ref{Prop:SkewZeroClosedGeneratenullvectors} gives that the vertex cover must also be skew zero forcing closed. A similar situation arises for hypertrees. Consider the following example (note that smaller examples exist), where the filled vertices denote a vertex cover of the given hypertree. 

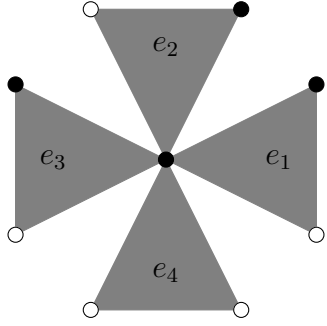
\begin{figure}[H]
\centering
\begin{tikzpicture}
\filldraw[gray] (0,0) -- (2,1) -- (2,-1) -- (0,0) -- (1,-2) -- (-1,-2) -- (0,0) -- (-2,-1) -- (-2,1) -- (0,0) -- (-1,2) -- (1,2) -- (0,0);

\filldraw[black](0,0)circle(0.1);
\filldraw[black](2,1)circle(0.1);
\filldraw[white](2,-1)circle(0.1);
\draw[black](2,-1)circle(0.1);
\filldraw[black](1,2)circle(0.1);
\filldraw[white](-1,2)circle(0.1);
\draw[black](-1,2)circle(0.1);
\filldraw[black](-2,1)circle(0.1);
\filldraw[white](-2,-1)circle(0.1);
\filldraw[white](-1,-2)circle(0.1);
\filldraw[white](1,-2)circle(0.1);
\draw[black](-2,-1)circle(0.1);
\draw[black](-1,-2)circle(0.1);
\draw[black](1,-2)circle(0.1);

\node at (1.5,0) {$e_1$};
\node at (0,1.5) {$e_2$};
\node at (-1.5,0) {$e_3$};
\node at (0,-1.5) {$e_4$};
\end{tikzpicture}
\caption{$3$-uniform hyperstar on four edges}
\label{fig:hypergraphEx}
\end{figure}
If $U$ is the set of filled vertices depicted in the previous figure and $v$ is the unique vertex of degree four, then $\phi_{U}f_{\mathcal{T}, v}$ contains exactly one monomial, namely the monomial corresponding to $e_4$. If a product of variables is zero, then at least one of the variables is zero, so $U$ is not the set of zero entries for any nullvector of this hypergraph. This observation leads to a hypergraph skew zero forcing rule (which differs from the hypergraph zero forcing rule presented as Def.~1.5 in \cite{Hog2020}). 
\begin{definition}\label{Def:SZFruleHyper}
Let $\mathcal{H}$ be a hypergraph. 
\begin{itemize}
\item A subset $Z \subseteq V(\mathcal{H})$ defines an initial coloring by filling all vertices of $Z$ and all other vertices remain unfilled. 
\item The skew zero forcing rule at a vertex $v$ says: If a vertex $v$ is incident to exactly one edge $e$ with no filled vertices in $e \setminus\{v\}$, then change the color of one vertex in $e\setminus \{v\}$ to filled. 
\item An (SZF-)derived set of an initial coloring $Z$ is the result of applying the skew zero forcing rule until no more changes are possible. 
\end{itemize} 
\end{definition}
In contrast to trees, the skew zero forcing  rule for hypergraphs, even just for hypertrees, does not necessarily generate a unique derived set. For example, if $\mathcal{T}$ is the $3$-uniform hyperedge with $V(\mathcal{T}) = \{u,v,w\}$ and $Z = \{u\}$, then $\{u,v\}$ and $\{u,w\}$ are both derived sets of $Z$. Similarly, if $Z = \emptyset$, then all elements of $\binom{V(\mathcal{T})}{2}$ are derived sets of $Z$. Thus, the ``skew zero forcing closure of $\emptyset$" that is considered in Theorem \ref{Thm:BigEquiv} is not well defined for hypertrees, as there are many sets derived from $\emptyset$.  However, we still sometimes refer to sets which are ``stalled'' in the sense that the SZF rule cannot be applied to them anywhere as ``SZF-closed sets''.

On the other hand, a reasonable choice of analogue to kernel-closed sets in this context is the family of zero loci of nullvectors.  Indeed, throughout the sequel, we refer to a set $S \subseteq V(\mathcal{T})$ as ``kernel-closed'' if it is the zero locus of a nullvector.  This raises the question if the family of SZF-closed vertex covers and the kernel-closed sets are the same. Unfortunately, they are not. Consider the hypergraph given in Figure \ref{fig:hypergraphEx}. Let $\mathcal{T}$ be the subhypertree given by edges $e_1$ and $e_2$. If $v$ is the vertex of degree $2$ in $\mathcal{T}$, the generating sets for irreducible components of $\ker(\mathcal{T})$ are $\{v\}$ and $V(\mathcal{T}) \setminus \{v\}$. However, $v$ together with a pendant vertex on each edge form a derived set of $\emptyset$. So, not all sets derived from $\emptyset$ generate irreducible components of $\ker(\mathcal{T})$. However, the converse is true, as is shown by the following result.

\begin{theorem}\label{thm:allcomponents} For each irreducible component $C$ of $\ker(\mathcal{T})$, there is an SZF-closed vertex cover $S \subseteq V(\mathcal{T})$ so that $C = \mathcal{V}^\mathcal{T}(S)$.
\end{theorem}

We postpone the proof to build a few useful tools.

\begin{prop}[Prop. 5.20 in \cite{Milne17}] \label{prop:cartesian} If $V$ and $W$ are irreducible affine varieties over an algebraically closed field, then $V \times W$ is as well.
\end{prop}

In fact, the way we will often use Proposition \ref{prop:cartesian} is: if $I \subset \mathbb{C}[x_1,\ldots,x_n]$ and $J \subset \mathbb{C}[y_1,\ldots,y_m]$ are prime ideals and $I',J'$ are the ideals they generate in $\mathbb{C}[x_1,\ldots,x_n,y_1,\ldots,y_m]$, respectively, then $I'+J'$ is also a prime ideal, and $\cV(I'+J') = \cV(I) \times \cV(J)$.  The following lemma establishes that some ideals of the form $L_{\mathcal{T}, V} = \langle \mathbf{x}_V \cup \{\phi_{V} f_{\mathcal{T}, v} : v \in V(\mathcal{T})\} \rangle$ are prime. 

\begin{lemma}\label{Lem:IisPrime}
Let $\mathcal{T}$ be a linear hypertree where each edge has rank at least three. For any $U \subseteq V(\mathcal{T})$ with $U$ an SZF-closed vertex cover of $\mathcal{T}$, $L_{\mathcal{T}, U}$ is a prime ideal, and $\mathbf{x}_U \cup \{\phi_{U} f_{\mathcal{T}, v} : v \in V(\mathcal{T})\}\setminus\{0\}$ is an irredundant set of generators for it.
\end{lemma}
\begin{proof}  
The generators of $L_{\mathcal{T}, U}$ are a finite collection of variables, as well as polynomials of a specific form: sums of monomials which are products of all but one vertex variable of an edge. 
Let $\mathcal{K} = \{\phi_{U} f_{\mathcal{T}, v} : v \in V(\mathcal{T})\} \setminus \{0\}$. 
Since $U$ is a vertex cover of $\mathcal{T}$, $\mathcal{K} = \{\phi_{U} f_{\mathcal{T}, v} : v \in V(\mathcal{T})\}\setminus\{0\} = \{\phi_{U} f_{\mathcal{T}, v} : v \in U\}\setminus\{0\}$.  Let $U' \subseteq U$ be any minimal set of $K:=|\mathcal{K}|$ vertices $v$ so that $\mathcal{K} = \{\phi_{U} f_{\mathcal{T}, v} : v \in U'\}\setminus\{0\}$. Since $U$ is SZF-closed, $U'$ contains no pendant vertices of $\mathcal{T}$. Let $\ell$ be a pendant vertex of $\mathcal{T}$, and label the elements of $U' = \{u_i\}_{i=1}^{K}$ so that if $i < j$, then $\dist(\ell, u_i) \leq \dist(\ell, u_j)$. 
We show by induction on $k \in \{0,\ldots,K\}$ that the ideal generated by $\mathbf{x}_U \cup \{\phi_U f_{\mathcal{T}, u_i}\}_{i=1}^k$ is prime, and that $\mathbf{x}_U \cup \{\phi_U f_{\mathcal{T}, u_i}\}_{i=1}^k$ is irredundant as generators. Then the base case $k = 0$ holds because $\mathbf{x}_U$, as just a collection of variables, generates a prime ideal and all such variables are necessary to generate $\langle \mathbf{x}_U \rangle$. Fix an integer $0 \leq k < K$ and suppose that the result holds for $\mathbf{x}_U \cup \{\phi_U f_{\mathcal{T}, u_i}\}_{i=1}^k$. Let $\mathcal{K}' = \mathbf{x}_U \cup \{\phi_U f_{\mathcal{T}, u_i}\}_{i=1}^{k+1}$ and define $g := \phi_U f_{\mathcal{T}, u_{k+1}}$. Since $\mathcal{T}$ is a tree and $\dist(\ell, u_i) \leq \dist(\ell, u_{k+1})$ for all $1\leq i\leq k$, some variables in $g$ do not appear as variables in $\mathbf{x}_U \cup \{\phi_U f_{\mathcal{T}, u_i}\}_{i=1}^k$, e.g., any variables of vertices incident to $u_{k+1}$ at distance $\dist(\ell, u_{k+1}) + 1$ from $\ell$, which exist since $u_{k+1}$ is not pendant in $\mathcal{T}$. 

Let $X$ be the set of variables of $g$ also occurring as variables of polynomials in $\mathcal{K}'\setminus\{g\}$. Define $Y$ to be the variables that appear in $g$ that are not contained in $X$, which is nonempty by the argument above. 
Additionally, let $Z$ be the collection of variables in polynomials of $\mathcal{K}'$ except the variables contained in $X$. Define a collection of new variables $X' := \{x_m' : x_m \in X\}$. Let the polynomial $g'$ be $g$ after application of the evaluation map that sends $x_m \mapsto x_m'$ for each $x_m \in X$. Let $I$ be the ideal generated by $\mathcal{K}'\setminus\{g\}$. The induction hypothesis gives that $I$ is prime. Note that $g = \phi_{U} f_{\mathcal{T}, u_{k+1}}$ is a sum of at least two monomials of positive degree by the skew-closedness of $U$. Thus, the ideal $\langle g'\rangle$ is prime because $g'$ is irreducible, since the linearity of $\mathcal{T}$ implies that no variable divides more than one monomial of $g'$. Proposition \ref{prop:cartesian} gives the primality of the ideal generated by $I+\langle g'\rangle$. Let $\sigma:\mathbb{C}[X\cup Z] \times \mathbb{C}[X'\cup Y] \to \mathbb{C}[X\cup Y\cup Z]$ be the quotient homomorphism $\sigma:f \mapsto f + \langle \{x_i - x_i' : x_i \in X\} \rangle$. Clearly, $\sigma$ is surjective, so Proposition $3.34$b in \cite{Milne17} (that surjective homomorphisms preserve primality) completes the proof of primality. 

Since $Y$ is nonempty, $g$ introduces a new variable in $\mathcal{K}'$.  Since $\mathcal{K}' \setminus \{g\}$ is irredundant by induction, we may conclude that $\mathcal{K}'$ is also irredundant.
\end{proof}

Notice then that if $\mathcal{T}$ is a linear hypertree where each edge has rank at least three and $U\subseteq V(\mathcal{T})$ is an SZF-closed vertex cover, then the codimension of $\mathcal{V}^{\mathcal{T}}(U)$ is $|U \cup \{\phi_{U} f_{\mathcal{T},v} : v \in V(\mathcal{T})\}\setminus\{0\}|$.

\begin{proof}[Proof of Theorem \ref{thm:allcomponents}]
If $\mathbf{x}$ is a nullvector of $\mathcal{T}$, then $Z(\mathbf{x})$ is an SZF-closed vertex cover of $\mathcal{V}(T)$ by Proposition \ref{Prop:vertex cover}. Therefore, Lemma \ref{Lem:IisPrime} gives that $\mathcal{V}^\mathcal{T}(Z(\mathbf{x}))$ is an irreducible variety. Note that $\mathbf{x} \in \mathcal{V}^\mathcal{T}(Z(\mathbf{x}))$, and $\mathcal{V}^\mathcal{T}(Z(\mathbf{x})) \subseteq \ker(\mathcal{T})$, so, 
$$
\ker(\mathcal{T}) = \bigcup_{\mathbf{x} \in \ker(\mathcal{T})} \mathcal{V}^\mathcal{T}(Z(\mathbf{x})).
$$
Thus, the minimal sets $\mathcal{V}^\mathcal{T}(Z(\mathbf{x}))$ are in bijection with the irreducible components of $\ker(\mathcal{T})$ (essentially because $\mathbb{C}[x_1,\ldots,x_n]$ is Noetherian), so $C = \mathcal{V}^\mathcal{T}(Z(\mathbf{x}))$ for some $\mathbf{x}$.
\end{proof}

\begin{lemma}\label{Lem:NotOneLightEdge}
Let $\mathcal{T}$ be a hypertree so that every edge has rank at least three. Let $U \subseteq V(\mathcal{T})$ be given so that $\mathcal{V}^{\mathcal{T}}(U)$ is an irreducible component of $\ker(\mathcal{T})$. Then for each $v \in V(\mathcal{T})$, $|\{e \in E(\mathcal{T}) : e \cap U = \{v\}\}| \neq 1$, i.e., the skew zero forcing rule stalls on $U$. 
\end{lemma}
\begin{proof}
If $|\{e \in E(\mathcal{T}) : e \cap U = \{v\}\}| = 1$, then $\phi_{U} f_{\mathcal{T}, v}$ is one monomial. Since every edge of $\mathcal{T}$ has rank at least three, $\phi_{U} f_{\mathcal{T}, v}$ is a product of at least two variables, contradicting that $\mathcal{V}^{\mathcal{T}}(U)$ forms an irreducible component of $\ker(\mathcal{T})$.
\end{proof}

\begin{lemma}\label{Lem:IrreducibleVariety}
Let $\mathcal{T}$ be a linear hypertree where each edge has rank at least three. Let $U \subseteq V(\mathcal{T})$. Then $\{\phi_{U}f_{\mathcal{T}, v} : v \in V(\mathcal{T})\}$ does not contain any polynomials with exactly one monomial if and only if $\mathcal{V}^{\mathcal{T}}(U)$ is irreducible if and only if $U$ is SZF-closed. 
\end{lemma}
\begin{proof}
The backward direction of the first equivalence is Lemma \ref{Lem:NotOneLightEdge}, since the number of nonzero monomials appearing in $\phi_{U}f_{\mathcal{T}, v}$ equals $|\{e \in E(\mathcal{T}) : e \cap U = \{v\}\}|$, which also implies the second equivalence in the statement.  This also implies that, if $\{\phi_{U}f_{\mathcal{T}, v} : v \in V(\mathcal{T})\}$ contains no single-monomial polynomials, then $U$ is SZF-closed.  The forward direction follows from Lemma \ref{Lem:IisPrime} once we  show that $U$ is a vertex cover. We proceed by induction on the size of $E(\mathcal{T})$. If $|E(\mathcal{T})| = 1$, then $U$ -- which is nonempty because it is SZF-closed -- is clearly a vertex cover of $\mathcal{T}$. Suppose the result holds for all $1\leq |E(\mathcal{T})| \leq k$, and let $|E(\mathcal{T})| = k+1$. 

Since $\mathcal{T}$ is a hypertree, $\mathcal{T}$ contains a leaf edge $\ell$. Let $v_\ell$ be a pendant vertex of $\ell$, meaning $f_{\mathcal{T}, v_\ell}$ contains one monomial. Thus, $U$ contains at least one element of $\ell$. Let $U_\ell = U \cap \ell$ and $A := \{e\in E(\mathcal{T}): e \cap U_\ell \neq\emptyset\}$. Further, let $\mathcal{T}'$ be the subhyperforest of $\mathcal{T}$ induced by the non-isolated vertices of $(V(\mathcal{T}),E(\mathcal{T}) \setminus A)$. If $v\in V(\mathcal{T}')$, then there exists $e \in E(\mathcal{T})$ so that $v\in e$ and $e\cap U_\ell = \emptyset$, meaning $v \notin U_\ell$. Thus, if $e' \in A$ is incident to $v$, all vertices of $e' \cap U_\ell$ are distinct from $v$. As a result, the monomial of $f_{\mathcal{T}, v}$ given by edge $e'$ does not appear in $\phi_U f_{\mathcal{T}, v}$.
Since no monomials corresponding to edges of $A$ appear in $\phi_U f_{\mathcal{T}, v}$ for any $v \in V(\mathcal{T}')$, $\{\phi_{U}f_{\mathcal{T}, v} : v \in V(\mathcal{T}')\} = \{\phi_{U\setminus U_\ell}f_{\mathcal{T}', v} : v \in V(\mathcal{T}')\}$. Thus, $\{\phi_{ U\setminus U_\ell}f_{\mathcal{T}', v} : v \in V(\mathcal{T}')\}$ does not contain any polynomials with exactly one monomial. Since $|E(\mathcal{T}')| < |E(\mathcal{T})|$, the induction hypothesis gives that $U \setminus U_\ell$ is a vertex cover of $\mathcal{T}'$. Therefore, since $U_\ell$ is a vertex cover of the edges in $A$ and $E(\mathcal{T}) = A \cup E(\mathcal{T}')$, $U$ is a vertex cover of $\mathcal{T}$, completing the proof. 
\end{proof}

\begin{prop}\label{Prop:SZFMonotone}
Let $\mathcal{T}$ be a linear hypertree. If $A, B\subseteq V(\mathcal{T})$ so that $A\subseteq B$ and $B' \subseteq V(\mathcal{T})$ is SZF-derived from $B$, then there exists $A'\subseteq V(\mathcal{T})$ so that $A'$ is SZF-derived from $A$ and $A'\subseteq B'$. 
\end{prop}
\begin{proof}
Let $A, B, B'\subseteq V(\mathcal{T})$ so that $A\subseteq B$ and $B'$ is SZF-derived from $B$. Let $\{v_i\}_{i=1}^l$, $\{u_i\}_{i=1}^l$ be sequences of vertices so that the SZF rule applied to $v_i$ forces $u_i$ for each $1\leq i\leq l$ when deriving $B'$ from $B$. Furthermore, define $\{v_i'\}_{i=1}^m, \{u_i'\}_{i=1}^m \subseteq V(\mathcal{T})$ to be maximal subsequences of $\{v_i\}$ and $\{u_i\}$ so that the SZF rule can be applied at $v_i'$ to force $u_i'$ starting from set $A$. Let $S = A\cup \{u_i'\}_{i=1}^m$. Since $A \subseteq B$, and $B' = B\cup \{u_i\}_{i=1}^l$, $S \subseteq B'$. If $S$ is SZF-closed, then defining $A' := S$ completes the proof. 

Suppose instead that $S$ is not SZF-closed. Then there exists vertex $w\in V(\mathcal{T})$ so that the SZF rule can be applied at $w$. Thus, $|\{e \ni w : (e\setminus \{w\}) \cap S = \emptyset\}| = 1$. Let $f$ be the unique edge containing $w$ so that $(f \setminus \{w\}) \cap S = \emptyset$. Since $S \subseteq B'$ and $B'$ is SZF-closed, $(f \setminus \{w\}) \cap B' \neq \emptyset$.  Let $z \in (f \setminus \{w\}) \cap B'$, and define $S_1 := S\cup \{z\}$, which is SZF-derived from $S$ by applying the SZF-rule at $w$.  Then $S_1 \subseteq B'$ and $S_1$ is SZF-derived from $A$.  However, this contradicts the maximality of $S$.
\end{proof}

\begin{prop}\label{Prop:GenSetsDerivedFromEmpty}
Let $\mathcal{T}$ be a linear hypertree so that every edge has rank at least three. Let $S\subseteq V(\mathcal{T})$ so that $\mathcal{V}^{\mathcal{T}}(S)$ is an irreducible component of $\ker(\mathcal{T})$ and $S$ is the intersection of all zero loci of elements of $\mathcal{V}^{\mathcal{T}}(S)$. Then $S$ is SZF-derived from $\emptyset$. 
\end{prop}
\begin{proof}
Let $S$ be as is given in the statement. By Proposition \ref{Prop:vertex cover}, $S$ is a vertex cover of $\mathcal{T}$. Additionally, Lemma \ref{Lem:IrreducibleVariety} gives that $S$ is derived from itself under the skew zero forcing rule, i.e., $S$ is SZF-closed. Let $A\subseteq S$ be a set of minimum size so that $S$ is SZF-derived from $A$. By way of contradiction, suppose $A\neq \emptyset$. Then, since $\emptyset \subseteq A$, Proposition \ref{Prop:SZFMonotone} gives the existence of SZF-closed set $A_0 \subseteq V(\mathcal{T})$ so that $A_0 \subseteq S$ and $A_0$ is SZF-derived from $\emptyset$. Notice that $A\neq \emptyset$ implies $A_0 \subsetneq S$. Since $A_0$ and $S$ are SZF-closed, they are the intersections of zero loci of vectors in $\mathcal{V}^\mathcal{T}(A_0)$ and $\mathcal{V}^\mathcal{T}(S)$ respectively. Thus, $\mathcal{V}^\mathcal{T}(S) \subsetneq \mathcal{V}^\mathcal{T}(A_0)$. Furthermore, since  $A_0$ is SZF-closed, Lemma \ref{Lem:IrreducibleVariety} implies $\mathcal{V}^\mathcal{T}(A_0)$ is an irreducible variety, contradicting that $\mathcal{V}^{\mathcal{T}}(S)$ is an irreducible component of $\mathcal{V}_0(\mathcal{T})$. Therefore, $A = \emptyset$. 
\end{proof}

\begin{cor}\label{Cor:MinimalSkewForcingEmptysetClosures}
Let $\mathcal{T}$ be a linear hypertree where every edge has cardinality at least $3$. Then SZF-closed sets form a poset under the inclusion relation, and minimal elements of this poset are exactly the generators for irreducible components of $\ker(\mathcal{T})$, and they are SZF-derived from $\emptyset$. 
\end{cor}

%HERE

The above Corollary indicates the origin of the term ``generating set'' to refer to minimal zero loci of nullvectors. Now we turn our attention to kernel-closed sets in linear hypertrees and their relation to SZF-closed sets. The following result mirrors that of Proposition \ref{Prop:SkewZeroClosedGeneratenullvectors} for trees, in that we show that SZF-closed set are the zero loci of individual nullvectors. 

\begin{theorem}\label{Thm:SZFimpliesKernelClosed}
Let $\mathcal{T}$ be a linear hypertree. %so that every edge has rank at least three. 
If $S \subseteq V(\mathcal{T})$ is SZF-closed, then $S$ is kernel-closed. 
\end{theorem}
\begin{proof}
Let $S \subseteq V(\mathcal{T})$ be SZF-closed. Consider rooting $\mathcal{T}$ at a pendant vertex $w$ of leaf edge $\ell$. We construct a nullvector $\mathbf{x}$ with entries $x_v$ for $v \in V(\mathcal{T})$ by iteratively working through the hypertree $\mathcal{T}$. Start by assigning zeros to coordinates corresponding to vertices of $S$, i.e., let $x_v = 0$ for each $v \in S$. Now we choose values for all nonzero coordinates of $\mathbf{x}$. 

Let $x_v = 1$ for every pendant vertex $v \in \ell \setminus S$. Since $S$ is SZF-closed, $\phi_S f_{\mathcal{T}, v} = 0$ for every pendant vertex $v \in \ell$. If $\mathcal{T}$ is a single edge, we have completed the proof. Otherwise, let $z \in \ell$ be the unique vertex satisfying $\deg(z) > 1$. If $z \in S$, then we already know $x_{z} = 0$. If not, let $x_{z} = 1$. Therefore, we have determined the entries of $\mathbf{x}$ for every vertex with distance at most one from $w$. 
%If $\ell \notin S$, let $x_\ell = 1$. Furthermore, let $v_1$ be the unique neighbor of $\ell$ (note that $v_1$ is the only vertex of $T$ at height one). Since $S$ is skew zero forcing closed, $v_1 \in S$, so $x_{v_1} = 0$. 

Let $v \in V(\mathcal{T})$ be distance $h \geq 1$ from $w$, and if $u \in V(\mathcal{T})$ satisfies $\dist(w, u) \leq h$, then $x_u$ has already been assigned. Furthermore, assume that if $u \in V(\mathcal{T})$ satisfies $\dist(u, w) < h$, then $f_{\mathcal{T}, u}(\mathbf{x}) = 0$. 

If every edge incident to $v$ contains a vertex of $S\setminus \{v\}$, then define $x_y = 1$ for every $y \in V(\mathcal{T}) \setminus S$ satisfying $\dist(w, y) = h+1$ and $y$ is adjacent to $v$. In this case, $\phi_S f_{\mathcal{T}, v} = 0$, so $f_{\mathcal{T}, u}(\mathbf{x}) = 0$ holds. If this is not the case, let edges $\{e_i\}_{i=1}^m$ contain no vertices of $S\setminus \{v\}$ for some $m\geq 2$. The bound on $m$ comes from the assumption that $S$ is SZF-closed. We split into two cases. 

Case $1$: There exists $1\leq j\leq m$ so that $e_j$ contains vertices at distance $h-1$ from $w$. Then all entries of $\mathbf{x}$ corresponding to vertices of $e_j$ have already been assigned. Define $c := \prod_{u \in e_j \setminus \{v\}} x_u$, and note that by assumption, $c \neq 0$. Choose $u_i \in e_i\setminus \{v\}$ for each $1\leq i\leq m$ so that $i\neq j$. Then define $x_{u_i} = -c / (m-1)$ and if $u \in \bigcup_{i\neq j} (e_i \setminus \{v, u_i\})$, then let $x_u = 1$. Lastly, if $u$ is a neighbor of $v$ so that $x_u$ has not been assigned (these are vertices outside $S$ and outside $\{e_i\}$), let $x_u = 1$. It is straightforward to see that $f_{\mathcal{T}, v}(\mathbf{x}) = 0$. 

Case $2$: There does not exist $1\leq j\leq m$ so that $e_j$ contains vertices at distance $h-1$ from $w$. For each $u \in e_1 \setminus \{v\}$, define $x_u = 1$. Choose $u_i \in e_i\setminus \{v\}$ for each $2\leq i\leq m$. Then define $x_{u_i} = -1 / (m-1)$ and if $u \in \bigcup_{i\geq 2} (e_i \setminus \{v, u_i\})$, then let $x_u = 1$. Lastly, if $u$ is a neighbor of $v$ so that $x_u$ has not been assigned (these are vertices outside $S$ and outside $\{e_i\}$), let $x_u = 1$. Again we have that $f_{\mathcal{T}, v}(\mathbf{x}) = 0$. 
\end{proof}

The last result of this section provides a hypergraph analogue to the statement (Proposition \ref{Prop:SkewZeroClosedGeneratenullvectors}) that trees are SZF-complete.

\begin{cor}
For linear hypertree $\mathcal{T}$ with every edge having rank at least three, a set $S\subseteq V(\mathcal{T})$ is SZF-closed if and only if it is kernel-closed. 
\end{cor}

\begin{proof}
    The only content of the statement beyond Theorem \ref{Thm:SZFimpliesKernelClosed} is that each  kernel-closed set $U = Z(\mathbf{x})$ is SZF-closed.  Suppose not; then $|\{e \in E(\mathcal{T}) : e \cap U = \{v\}\}| = 1$ for some $v \in V(\mathcal{T})$, so $\phi_{U} f_{\mathcal{T}, v}$ is one monomial which evaluates to a nonzero value at $\mathbf{x}$, contradicting that $\mathbf{x}$ is a nullvector.
\end{proof}

\subsection{Complete Hypergraph Nullvariety}

In this section, we investigate the nullvariety and SZF-closed sets of complete hypergraphs. We let $\mathcal{K}_n^{(k)}$ denote the $k$-uniform complete hypergraph on $n$ vertices. Let $\mathcal{X}$ be a collection of variables. Then define $e_k(\mathcal{X})$ to be the $k$th elementary symmetric polynomial in the variables of $\mathcal{X}$, i.e., if $\mathcal{X} = \{x_i\}_{i=1}^m$, then 
$$
e_k(\mathcal{X}) = \sum_{s \in \binom{[m]}{k}} \prod_{i \in s} x_i.
$$
One important property of the complete hypergraph $\mathcal{K}_n^{(k)} = ([n],\binom{[n]}{k})$, is the following equality, where we assume $v \in [n]$. 
$$
f_{\mathcal{K}_n^{(k)}, v} = \frac{\partial}{\partial x_v} e_k\left(\{x_u\}_{u\in [n]}\right)
$$
We exploit this in the following result. The proof is a formalization of a sketch given by \cite{Mec2017}. 

\begin{theorem}\label{Thm:CompleteHypergraphIrrComps}
If $n \geq k\geq 2$, let $\mathcal{G} = \mathcal{K}_n^{(k)}$. Then the collection of irreducible components of $\ker(\mathcal{G})$ is given by 
$$
\left\{ \mathcal{V}^{\mathcal{G}}(S) : S \in \binom{[n]}{n - k + 2} \right\}_.
$$
\end{theorem}
\begin{proof}
Let $\mathbf{y} \in \ker(\mathcal{G})$, and let $\mathcal{X} = \{x_u\}_{u\in [n]}$. Since $f_{\mathcal{G}, u}(\mathbf{y}) = 0$ for all $u \in [n]$, we are interested in the simultaneous vanishing of the following set. 
$$
\left\{ \frac{\partial}{\partial x_u} e_k(\mathcal{X})\right\}_{u \in U}
$$
Note that for every $S \in \binom{[n]}{n - k + 2}$, we have $\mathcal{V}^{\mathcal{G}}(S)\subseteq \ker(\mathcal{G})$, as for any $\mathbf{z} \in \mathcal{V}(S)$, every monomial of $\frac{\partial}{\partial x_u} e_k(\mathcal{X})$ evaluates to zero by the pigeonhole principle. Thus, it suffices to show that at least $n-k+2$ coordinates of $\mathbf{y}$ are zero. We prove the result by induction on $k$ for fixed $n$. Let $n\geq 2$ be given. If $k = 2$, then $\frac{\partial}{\partial x_i} e_k(\mathcal{X}) = \left( \sum_{j\in [n]} x_j \right) - x_i$. Thus, if $n = 2$, then $y_1 = y_2 = 0$. Otherwise, if $n > 2$, then $0 = (\frac{\partial}{\partial x_i} e_k(\mathcal{X}))(\mathbf{y}) - (\frac{\partial}{\partial x_j} e_k(\mathcal{X}))(\mathbf{y}) = y_j - y_i$ for each distinct pair $i, j\in [n]$. Therefore, $\mathbf{y}$ is a constant vector, so $\mathbf{y} = \mathbf{0}$, giving the desired result in this case. 

Now suppose the result holds for some $k\geq 2$ so that $k < n$, and for all $n'$ so that $k\leq n'\leq n$. Now we consider $\mathcal{G}':=\mathcal{K}_n^{(k+1)}$. Let $\mathbf{y}' \in \ker(\mathcal{G}')$. 

Take note of the following equality: 
$$
\frac{\partial}{\partial x_u} e_{k+1}(\mathcal{X}) = e_{k}(\mathcal{X}) - x_u \cdot \frac{\partial}{\partial x_u} e_{k}(\mathcal{X})
$$

Notice also that 
\begin{equation}\label{Eq:SymPolyId}
e_{k}(\mathcal{X}) = \frac{1}{k} \sum_{u \in [n]} x_u \frac{\partial}{\partial x_u} e_{k}(\mathcal{X}).
\end{equation}
Thus, $\frac{\partial}{\partial x_u} e_{k+1}(\mathcal{X}) = (e_{k}(\mathcal{X}))(\mathbf{y}') - (x_u \cdot \frac{\partial}{\partial x_u} e_{k}(\mathcal{X}))(\mathbf{y}') = 0$ for all $u \in [n]$ implies that every $(x_u \cdot \frac{\partial}{\partial x_u} e_{k}(\mathcal{X}))(\mathbf{y}') = 0$, as otherwise contradicts (\ref{Eq:SymPolyId}). We seek to show that at least $n-(k+1)+2$ coordinates of $\mathbf{y}'$ are zero. Suppose that $q \in \mathbb{N}$ coordinates of $\mathbf{y}'$ are zero and the remaining $n - q$ are nonzero. For the $n-q$ coordinates $u$ which are nonzero, $(x_u \cdot \frac{\partial}{\partial x_u} e_{k}(\mathcal{X}))(\mathbf{y}') = 0$ implies $( \frac{\partial}{\partial x_u} e_{k}(\mathcal{X}))(\mathbf{y}') = 0$. Therefore, if $n - q \geq k$, this contradicts the induction hypothesis for $n' = n-q$. Thus, $n-q \leq k-1$, giving $q \geq n-k+1 = n - (k+1) + 2$, completing the proof. 
\end{proof}

\begin{cor}\label{Cor:CompleteHypergraphKernelClosed}
If $n\geq k\geq 2$, then $S \subseteq V(\mathcal{K}_n^{(k)})$ is kernel-closed if and only if $|S| \geq n-k+2$. 
\end{cor}
\begin{proof}
Let $S\subseteq V(\mathcal{K}_n^{(k)})$. If $|S| < n-k+2$, then $S$ is the not the zero locus of a nullvector, as otherwise would contradict Theorem \ref{Thm:CompleteHypergraphIrrComps}. Conversely, if $|S| \geq n-k+2$, there are at most $k-2$ vertices outside $S$. Since $\mathcal{K}_n^{(k)}$ is $k$-uniform, $S$ contains at least two vertices from every edge of $E(\mathcal{K}_n^{(k)})$. Therefore, if $v \in V(\mathcal{K}_n^{(k)})$, then $\phi_S f_{\mathcal{K}_n^{(k)}, v} = 0$.
\end{proof}

Theorem \ref{Thm:CompleteHypergraphIrrComps} gives that $\ker({\mathcal{K}_{n}^{(k)}})$ contains $\binom{n}{n-k+2} = \binom{n}{k-2}$ irreducible components each of codimension $n-k+2$, i.e., dimension $k - 2$. Furthermore, Corollary \ref{Cor:CompleteHypergraphKernelClosed} completely describes the collection of kernel-closed sets for $\mathcal{K}_n^{(k)}$. What about the collection of SZF-closed sets? Those are summarized by the following result. 

\begin{prop}
Let $\mathcal{K}_{n}^{(k)}$ be a complete $k$-uniform hypergraph for some $2 \leq k\leq n$ and $U \subseteq [n] = V(\mathcal{K}_n^{(k)})$. Then $U$ is a stalled set if and only if $|U| \notin \{n-k, n-k+1\}$. 
\end{prop}
\begin{proof}
Let $U\subseteq [n]$ and $v \in [n]$ be arbitrary. We split into cases according to $|U|$. 

Case $1$: $|U| > n-k+1$. Then every edge of $\mathcal{K}_n^{(k)}$ incident to $v$ contains at least two representatives of $U$, so the SZF rule cannot be applied at $v$. 

Case $2$: $|U| = n-k+1$. Since $n\geq k$, $U\neq\emptyset$. If $v \in U$, then the fact that $\mathcal{K}_n^{(k)}$ is $k$-uniform and $V(\mathcal{K}_n^{(k)})\setminus U$ contains exactly $k-1$ vertices implies that there is exactly one edge $e$ incident to $v$ so that $(e\setminus \{v\}) \cap U = \emptyset$ (namely, the edge containing $v$ and all vertices outside $U$). Thus, $U$ is not stalled as the SZF rule can be applied at $v$. 

Case $3$: $|U| = n-k$. Since $k \geq 2$, $n-k < n$, so $U \neq V(\mathcal{K}_n^{(k)})$. If $v \notin U$, then the fact that $\mathcal{K}_n^{(k)}$ is $k$-uniform and $V(\mathcal{K}_n^{(k)})\setminus U$ contains exactly $k$ vertices implies that exactly one edge $e$ incident to $v$ satisfies $(e\setminus \{v\}) \cap U = \emptyset$ (namely the edge containing all vertices outside $U$, which contains $v$). Thus, $U$ is not stalled as the SZF rule can be applied at $v$. 

Case $4$: $|U| < n-k$. If $v \in U$, then the fact that $\mathcal{K}_n^{(k)}$ is $k$-uniform and $V(\mathcal{K}_n^{(k)})\setminus U$ contains at least $k+1$ vertices implies that at least $\binom{k+1}{k-1} > 1$ edges $e$ incident to $v$ satisfy $(e\setminus \{v\}) \cap U = \emptyset$. So, the SZF rule cannot be applied to $v$. On the other hand, if $v \notin U$, then the fact that $\mathcal{K}_n^{(k)}$ is $k$-uniform and $V(\mathcal{K}_n^{(k)})\setminus U$ contains at least $k+1$ vertices implies that at least $\binom{k}{k-1} > 1$ edges $e$ incident to $v$ satisfy $(e\setminus \{v\}) \cap U = \emptyset$. So, the SZF rule cannot be applied at $v$.
\end{proof}

The previous proposition shows the SZF-closed sets can be partitioned into two classes, those with size at least $n-k+2$ and those with size at most $n-k-1$. The partition class with larger sets agrees with the kernel-closed, while those in the other partition class are not kernel-closed. However, if $n=k$, the partition class containing sets of size $n-k-1$ does not exist, meaning the collections agree only for the single edge, $\mathcal{K}_k^{(k)}$. Thus, complete hypergraphs satisfying $n > k$ provide an example where the collections of SZF-closed and kernel-closed sets are not the same.

\section{Future Directions} \label{Sec:Conclusion}

Here, we list a few problems that remain open.

\begin{enumerate}
    \item Is there a combinatorial characterization of graphs which are SZF-complete, at least for bipartite graphs?  What about for hypergraphs?
    \item Describe the class of graphs whose SZF-closure is a matroid closure operator.
    \item How can the rest of Theorem \ref{Thm:BigEquiv} be generalized to hypertrees?
    \item Describe the structure of the kernel and SZF matroids.  What are the atoms, coatoms, independent/dependent sets, cycles, hyperplanes, etc.?
    \item Can the frameworks of closure operators and matroids, or generalizations thereof, be applied to describe the kernel and SZF-closed set systems for hypergraphs?
    \item Is Corollary \ref{Cor:Gammoid} true for a broader range of graphs $G$?
\end{enumerate}

\section{Acknowledgements}

Thanks to Alex Duncan and Darren Narayan for helpful discussions during the current work, to Vladimir Nikiforov for asking the right questions, and to Leslie Hogben for giving  fascinating talks that sparked our interest in zero forcing.

\bibliographystyle{plain}
\bibliography{ref}

\begin{thebibliography}{10}

\bibitem{Aaz2008}
Ashkan Aazami.
\newblock Hardness results and approximation algorithms for some problems on
  graphs, phd thesis.
\newblock {\em UWSpace}, 2008.

\bibitem{AkbKir2007}
S.~Akbari and S.J. Kirkland.
\newblock On unimodular graphs.
\newblock {\em Linear Algebra and its Applications}, 421(1):3--15, 2007.
\newblock Special Issue devoted to the 12th ILAS Conference.

\bibitem{AnsJacPenSaa2016}
Thomas Ansill, Bonnie Jacob, Jaime Penzellna, and Daniel Saavedra.
\newblock Failed skew zero forcing on a graph.
\newblock {\em Linear Algebra and its Applications}, 509:40--63, 2016.

\bibitem{Chu97}
Fan R.~K. Chung.
\newblock {\em Spectral graph theory}, volume~92 of {\em CBMS Regional
  Conference Series in Mathematics}.
\newblock Published for the Conference Board of the Mathematical Sciences,
  Washington, DC; by the American Mathematical Society, Providence, RI, 1997.

\bibitem{CoDu12}
Joshua Cooper and Aaron Dutle.
\newblock Spectra of uniform hypergraphs.
\newblock {\em Linear Algebra Appl.}, 436(9):3268--3292, 2012.

\bibitem{CooFic2021}
Joshua Cooper and Grant Fickes.
\newblock Geometric vs algebraic nullity for hyperpaths.
\newblock {\em PAMQ}, 18(6):2433--2460, 2022.

\bibitem{CveGut1972}
Drago\v{s}~M. Cvetkovi\'{c} and Ivan~M. Gutman.
\newblock The algebraic multiplicity of the number zero in the spectrum of a
  bipartite graph.
\newblock {\em Mat. Vesnik}, 9(24):141--150, 1972.

\bibitem{DeA2014}
Luz~M. DeAlba.
\newblock Some results on minimum skew zero forcing sets, and skew zero forcing
  number, 2014.

\bibitem{DulMen1958}
A.~L. Dulmage and N.~S. Mendelsohn.
\newblock Coverings of bipartite graphs.
\newblock {\em Canadian Journal of Mathematics}, 10:517–534, 1958.

\bibitem{FalMeaYan2016}
Shaun Fallat, Karen Meagher, and Boting Yang.
\newblock On the complexity of the positive semidefinite zero forcing number.
\newblock {\em Linear Algebra and its Applications}, 491:101--122, 2016.
\newblock Proceedings of the 19th ILAS Conference, Seoul, South Korea 2014.

\bibitem{FanBaoHua19}
Yi-Zheng Fan, Yan-Hong Bao, and Tao Huang.
\newblock Eigenvariety of nonnegative symmetric weakly irreducible tensors
  associated with spectral radius and its application to hypergraphs.
\newblock {\em Linear Algebra Appl.}, 564:72--94, 2019.

\bibitem{God1985}
C.~D. Godsil.
\newblock Inverses of trees.
\newblock {\em Combinatorica}, 5:33--39, 1985.

\bibitem{GutBor2011}
Ivan Gutman and Bojana Borovi\'{c}anin.
\newblock Nullity of graphs: an updated survey.
\newblock {\em Zb. Rad. (Beogr.)}, 14(22)(Selected topics on applications of
  graph spectra):137--154, 2011.

\bibitem{HarPlu1967}
Frank Harary and Michael~D. Plummer.
\newblock On the core of a graph.
\newblock {\em Proceedings of the London Mathematical Society},
  s3-17(2):305--314, 1967.

\bibitem{Hog2020}
Leslie Hogben.
\newblock Zero forcing and maximum nullity for hypergraphs.
\newblock {\em Discrete Applied Mathematics}, 282:122--135, 2020.

\bibitem{HoLiSh22}
Leslie Hogben, Jephian C.-H. Lin, and Bryan~L. Shader.
\newblock {\em Inverse problems and zero forcing for graphs}, volume 270 of
  {\em Mathematical Surveys and Monographs}.
\newblock American Mathematical Society, Providence, RI, 2022.

\bibitem{Mec2017}
Izaak~Meckler (https://mathoverflow.net/users/97414/izaak meckler).
\newblock Singular locus of zero set of elementary symmetric polynomial.
\newblock MathOverflow.
\newblock URL:https://mathoverflow.net/q/264226 (version: 2017-03-09).

\bibitem{IMAISU2010}
{IMA-ISU research group on minimum rank}.
\newblock Minimum rank of skew-symmetric matrices described by a graph.
\newblock {\em Linear Algebra and its Applications}, 432(10):2457--2472, 2010.

\bibitem{LoPl09}
L\'{a}szl\'{o} Lov\'{a}sz and Michael~D. Plummer.
\newblock {\em Matching theory}.
\newblock AMS Chelsea Publishing, Providence, RI, 2009.
\newblock Corrected reprint of the 1986 original [MR0859549].

\bibitem{Milne17}
James~S. Milne.
\newblock Algebraic geometry (v6.02), 2017.
\newblock Available at www.jmilne.org/math/.

\bibitem{Pul1995}
W.~R. Pulleyblank.
\newblock Matchings and extensions.
\newblock In {\em Handbook of Combinatorics (vo1. 1)}, chapter~3, pages
  176--232. MIT Press, 1995.

\bibitem{Qi05}
Liqun Qi.
\newblock Eigenvalues of a real supersymmetric tensor.
\newblock {\em J. Symbolic Comput.}, 40(6):1302--1324, 2005.

\bibitem{QiLu17}
Liqun Qi and Ziyan Luo.
\newblock {\em Tensor analysis}.
\newblock Society for Industrial and Applied Mathematics, Philadelphia, PA,
  2017.
\newblock Spectral theory and special tensors.

\bibitem{RobWel1980}
G.~C. Robinson and D.~J.~A. Welsh.
\newblock The computational complexity of matroid algorithms.
\newblock {\em Math. Proc. Comb. Phil. Soc.}, 87:29--45, 1980.

\bibitem{SaSa09}
T.~Sander and J.~W. Sander.
\newblock Tree decomposition by eigenvectors.
\newblock {\em Linear Algebra Appl.}, 430(1):133--144, 2009.

\bibitem{SciMifBor2021}
Irene Sciriha, Xandru Mifsud, and James~L. Borg.
\newblock Nullspace vertex partition in graphs.
\newblock {\em Journal of Combinatorial Optimization}, 42:310--326, 2021.

\bibitem{Shi2017}
Yaroslav Shitov.
\newblock On the complexity of failed zero forcing.
\newblock {\em Theoretical Computer Science}, 660:102--104, 2017.

\bibitem{WanShaYua2015}
Xiumei Wang, Weiping Shang, and Jinjiang Yuan.
\newblock On graphs with a unique perfect matching.
\newblock {\em Graphs and Combinatorics}, 31:1765--1777, 2015.

\bibitem{We95}
D.~J.~A. Welsh.
\newblock Matroids: fundamental concepts.
\newblock In {\em Handbook of combinatorics, {V}ol. 1, 2}, pages 481--526.
  Elsevier Sci. B. V., Amsterdam, 1995.

\end{thebibliography}

\end{document}